\documentclass{elsarticle}

\usepackage{oppapers}

\usepackage{tikz}
\usetikzlibrary{snakes}
\usetikzlibrary{positioning}

\usepackage[matrix,cmtip,arrow,curve]{xy}

\usepackage[unicode]{hyperref}

\usepackage{undertilde}

\newcommand{\cotor}{\operatorname{Cotor}}
\renewcommand{\hd}{d_v}
\renewcommand{\cd}{d_h}
\renewcommand{\exte}{u}
\begin{document}

\begin{frontmatter}

%\title{Homology operations and cosimplicial infinite loop spaces. I}
\title{Operations in the homology spectral sequence of a cosimplicial infinite loop space}
\author{Philip Hackney\fnref{fn}}
\ead{hackney@math.ucr.edu}
\address{Department of Mathematics, University of California, Riverside, 900 University Avenue,
Riverside, CA 92521, USA}
\fntext[fn]{Phone: 951-827-5402, Fax: 951-827-7314}

\begin{abstract}
Consider the mod 2 homology spectral sequence associated to a cosimplicial space $X$. We construct external operations whose target is the spectral sequence associated to $E\Sigma_2 \times_{\Sigma_2} (X\times X)$. If $X$ is a cosimplicial $E_\infty$-space, we couple these external operations with the structure map \[E\Sigma_2 \times_{\Sigma_2} (X\times X) \to X\] to produce internal operations in the spectral sequence. 
In the sequel we show that they agree with the usual Araki-Kudo operations on the abutment $H_*(\Tot X)$.
\end{abstract}


\end{frontmatter}

%\maketitle

%\tableofcontents

\section{Introduction}
\noindent Let $\msc$ be a fixed $E_\infty$ operad in spaces. The aim of this work is to provide a new proof of the following theorem.
\begin{thm}\label{T:turner}
Suppose that $X$ is a cosimplicial object in the category of $\msc$-spaces. Then there are operations in the mod-2 homology spectral sequence associated to $X$:
\begin{align*}
Q^{m}:E^r_{-s,t} &\to E^r_{-s,m+t} & m&\geq t \\
Q^m:E^r_{-s,t} &\to E^w_{m-s-t,2t} & m&\in[t-s,t]\end{align*}
where $w\in [r,2r-2]$  is given  by 
\[w = \begin{cases} r & m=t-s \\ 2r-2 & m\in [t-s+1,t-r+2] \\ r+t-m & m\in [t-r+3, t]. \end{cases} \]
\end{thm}
The first proof of this was given by Jim Turner in \cite{turner}. We present here a fundamentally different proof which is more direct and is amenable to generalization.
In particular, in \cite{me2b}, we prove the analogue of Theorem~\ref{T:turner} for cosimplicial $E_{n+1}$-spaces. 
The calculations for the $E_\infty$-case are substantially more transparent than those for the $E_{n+1}$ case, and \cite{me2b} requires much of the machinery and many of the calculations in this paper. 
Had we bundled this into one paper, we would have had to prove the $E_\infty$ case first anyway, so for simplicity we have separated it out.
Another motivation for providing a new proof of \ref{T:turner} is to address the issue of compatibility with Araki-Kudo operations\footnote{These are also referred to as Dyer-Lashof operations.} in the abutment of the spectral sequence, $H_*(\Tot X)$. Thus this paper also feeds into \cite{me2a}, where some of the themes (small examples and external operations) help us to give the first complete proof of \cite[5.11]{turner} (the proof sketch given there has substantial gaps).

Pictorially, the images of these operations applied to an element in bidegree $(-s,t)$ lie on the solid and dotted lines of Figure~\ref{F:e2turner}. We call operations of the first type ``vertical operations'' and operations of the second type ``horizontal operations''.
\begin{figure}[ht]
\centering
\begin{tikzpicture}
    %Axes
%    \draw[->] (0.2,0) -- (-5,0);
%	\draw[->] (-.65,0) -- (-5,0);		
%	\draw (0.2,0) -- (-.35,0);		
%	\draw (-.75,-.1) -- (-.55,.1);	%new
%	\draw (-.45,-.1) -- (-.25,.1);	%new
%    \draw[->] (0,-0.2) -- (0,5);
	\draw (2.2,0) -- (-5,0);		
    \draw[->] (2,-0.2) -- (2,5);
    %Horizontal ticks
    \draw (-1,2pt) -- (-1,-2pt) node[below] {$-s$};
%    \draw (-2.5,2pt) -- (-2.5,-2pt) node[below] {$s+r$};
    \draw (-4,2pt) -- (-4,-2pt) node[below] {$-2s$};
%    \draw (-5.5,2pt) -- (-5.5,-2pt) node[below] {$2s+r$};
%    \draw (-7,2pt) -- (-7,-2pt) node[below] {$2s+2r$};
    %Vertical Ticks
%    \draw (2pt,1) -- (-2pt,1) node[right] {$2t$};
    \draw (1.9,1) -- (2.1,1) node[right] {$2t$};
        \draw (1.9,.5) -- (2.1,.5) node[right] {$t$};
%
%    \draw (-2pt,1) -- (2pt,1) node[right] {$2t$};
%    \draw (-2pt,.5) -- (2pt,.5) node[right] {$t$};
%    \draw (2pt,2) -- (-2pt,2) node[right] {$2t+r-1$};
%    \draw (2pt,3) -- (-2pt,3) node[right] {$2t+2r-2$};
    %Points
    \fill (-4,1) circle (1.5pt) node[left] {\scriptsize $Q^{t-s}(x)$};;
    \fill (-1,1) circle (1.5pt) node[right] {\scriptsize $Q^t(x)$};
    \fill (-1,.5) circle (1.5pt) node[right] {\scriptsize $x$};
    %Arrow
    \draw [very thick, dotted] (-4,1) -- (-1,1);
    \draw [->,very thick] (-1,1) -- (-1,5);
%    \draw [->,very thick] (-7,3) -- (-2.5,3) -- (-2.5,5);
\end{tikzpicture} 
\caption{Vertical and Horizontal Operations}\label{F:e2turner}
\end{figure}
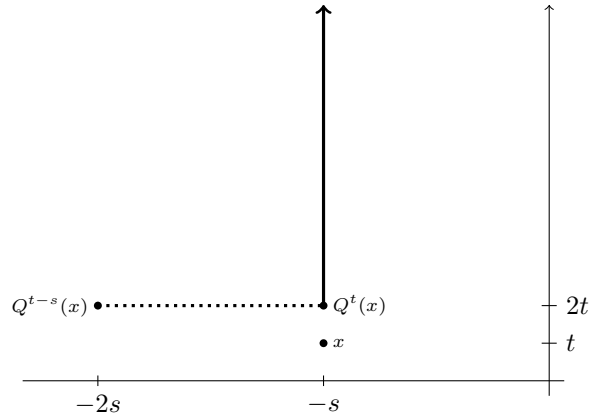

Prior to Turner's work, Araki-Kudo operations were constructed in Eilenberg-Moore spectral sequences in \cite{bahri} and \cite{ligaardmadsen}, though in both cases only the `vertical' operations were given. The theory of Steenrod operations, on the other hand, has a rich literature. Appropriate starting points are \cite{dwyer} and \cite[Chapter 7]{singer}. Steenrod operations in spectral sequences have been used to facilitate computations of cohomology; in particular, this extra structure sometimes allows one to prove collapse theorems. As an example, the existence of Steenrod operations was used in  \cite{mimuramori} to calculate the cohomology of the classifying spaces of some exceptional Lie groups.

We now mention two examples of cosimplicial $\msc$-spaces where the homology spectral sequence and its target are of interest. The theory presented in this sequence of papers may provide insight to each. For the first, consider a pullback diagram 
\[ \xymatrix{
E_f \ar@{->}[r] \ar@{->}[d] & E \ar@{->}[d] \\
Z \ar@{->}[r]^f & B 
}\]
and let $X^\bullet$ be the two-sided geometric cobar construction  with $X^n = Z \times B^n \times E$ (see \cite[\S 2]{rector} for a complete definition). Using $H_*$ to denote homology mod 2,  the homology spectral sequence associated to $X^\bullet$ is the Eilenberg-Moore spectral sequence
\[ E^2 \cong \cotor^{H_*(B)}(H_*(E), H_*(Z)) \Rightarrow H_*(E_f), \]
as demonstrated in \cite[Theorem 5.1]{rector}.
If  the above pullback diagram is a pullback diagram in the category of $\msc$-spaces, then $X^\bullet$ is a cosimplicial $\msc$-space. Thus, in this case, there are operations in the Eilenberg-Moore spectral sequence associated to the pullback diagram above.

As a second example, suppose that $Z$ is a space 
%(where `space' means `simplicial set' for the moment) 
and that $R$ is a commutative ring. In \cite[Chapter I]{bkbook}, Bousfield and Kan define the $R$-completion of $Z$ as the totalization of a certain cosimplicial space $\utilde{R} Z$. It turns out that if $Z$ is a $\msc$-space, then $\utilde{R} Z$ is a cosimplicial $\msc$-space. Thus there are homology operations in the associated spectral sequence, which, coupled with convergence results of Bousfield and Shipley \cite{bousfield, shipley}, may help provide information about the mod 2 homology of $R$-completions of various $\msc$-spaces.

The scheme of the paper is as follows. We work with simple \emph{Bousfield-Kan universal examples}, which one might think of as cosimplicial spheres. We construct \emph{external operations} for these examples and use the universal property to transport these operations into the spectral sequence for an arbitrary cosimplicial space. When that cosimplicial space is actually a cosimplicial $\msc$-space, we then obtain internal operations by combining the external operations with the $\msc(2)$-structure.

We begin with some background and notation.

\subsection{Homology Spectral Sequence and Passage to Chains}\label{S:homss}

Let $X$ be a cosimplicial space (where `space' means either topological space or simplicial set). We briefly outline the construction of the homology spectral sequence associated to $X$ (see \cite{bousfield} for more details). 

\begin{defn}
We always work over the field $\K = \Z /2$ and just write $\ch$ for the category of chain complexes over $\K$. 
For a set $Z$ with basepoint $*$, the notation $\K Z$ will always mean the free $\K$-module with basis $Z\setminus \set{*}$, so that $\K$ defines a functor $\Set_* \to \K\operatorname{Mod}$. If $Z$ is just a set, then $\K Z$ will be denote the free $\K$-module with basis $Z$.
%REF is this the right place to put this ?
\end{defn}

We will use the normalization functor
$ N: \K \operatorname{Mod}^{\Delta^{op}} \to \ch $
given, on a simplicial $\K$-module $Z$, by
\[ N_k(Z) = Z_k / (s_0 Z_{k-1} + \dots s_{k-1} Z_{k-1}) \]
with differential induced from the sum of the face maps.
Let
\[ S_*: \text{Spaces} \to \ch \] be the mod-2 (normalized) chains functor. The first step in the construction of the spectral sequence is to pass from $X$ to the cosimplicial chain complex $S_*(X)$.

If $Y$ is a cosimplicial object in an abelian category $\msa$, we will write $CY$ for the conormalization of $Y$
\begin{equation} CY^p = \coker \Big(\bigoplus_{k=1}^p d^k: \bigoplus Y^{p-1} \to Y^p \Big), \label{E:conorm}\end{equation} 
which is an object in $\coch^{\geq 0} \msa$, the category of nonnegative cochain complexes over $\msa$. 
The differential on $CY$, which we write as $\cd$, is the map induced by the coface map $d^0$.
When $\msa=\ch$, the category of chain complexes over $\K$, we will regard $CY$ as a left-plane bicomplex which consists of the $\K$-module $CY_q^p$ in bidegree $(-p,q)$.

\begin{defn}
In this setting, we always write \[ \hd: CY_q^p \to CY_{q-1}^p \] for the differential coming from $Y^p$ and
\[ \cd: CY_q^p \to CY_q^{p+1} \] for the differential which is induced from the coface maps. Given the definition of conormalization, we have that $\cd$ is just the coface map $d^0$. We will also write $\partial = \hd + \cd$ for the differential in the total complex below.
\end{defn}

Given a bicomplex $B$, we will let $\Tot B$\label{SYM:producttot} denote the product total complex:
\begin{equation} (\Tot B)_m = \prod_{j} B_{j,m-j}. \label{E:bitot} \end{equation}The appropriate filtration in this situation is the one by columns
\[ F^{k}_m = \prod_{j\leq k} B_{j,m-j}. \] 
We may regard $\Tot C(Y)$ as a subcomplex of
\[ \prod_m \Hom (\Delta_*^m, Y^m), \]
where this is the internal $\Hom$ in the category $\ch$ and $\Delta_*^m$ denotes normalized simplicial chains on the simplicial set $\Delta^m = \Hom(-,[m])$. The natural filtration of $\Delta_*^\bullet$ by skeleta induces the above filtration on $\Tot C(Y)$ (see \cite{bousfield}).

The homology spectral sequence associated to a cosimplicial space $X$ is, by definition, the one obtained from this filtration on  $\Tot CS_*(X)$. For this reason we usually work with cosimplicial chain complexes rather than cosimplicial spaces, though of course we will have to check that various geometric constructions we make behave well when we pass to chains.
This will usually take the form of an $E^1$ or $E^2$ isomorphism between  algebraic and geometric constructions.

\begin{defn}\label{NOTATION:SS}
If $C$ is a filtered chain complex with $F^{-s-1} \ci F^{-s}$, we will write 
\begin{align*}
Z_{-s,t}^r &= \setm{x\in F^{-s} C_{t-s}}{\partial x \in F^{-s-r} C} \\
B_{-s,t}^r &= \partial Z^{r-1}_{-s+r-1, t-r+2} + Z^{r-1}_{-s-1, t+1}
\end{align*} for the $r$-cycles and $r$-boundaries, and
$E_{-s,t}^r = Z_{-s,t}^r / B_{-s,t}^r$ for the $(-s,t)$ position of the $r^{\text{th}}$ page of the spectral sequence . Furthermore, if $x$ and $y$ are in $F^{-s} C_{t-s}$, then we will write $x\sim_r y$ if $x-y \in B_{-s,t}^r$.
\end{defn}

\subsection{External Operations} Araki-Kudo operations were introduced in \cite{kudoaraki}. 
In this section we recall, following the construction in \cite{may}, 
that the operations for a $\msc$-space $X$ may be obtained by combining an ``external operation'' with the $\msc(2)$-structure map.

Set $\pi := \Sigma_2 = \set{e,\sigma}$ and let $\triv$ be $\K$ with trivial $\pi$-action.   Let $W$\label{SYM:W} be the minimal $\K\pi$-free resolution of $\triv$, which is defined by
\begin{equation} W_i = \begin{cases} \K \pi \cdot e_i & i\geq 0 \\ 0 & i< 0 \end{cases} \label{E:W} \end{equation}
and 
\[ d (e_i) = (1+\sigma) e_{i-1}. \] This is $\K\pi$-chain-homotopic to $S_*(E\pi)$, which, combined with the shuffle map, gives a quasi-isomorphism
\[ W\tp (S_*(X)\ten S_*(X)) \to S_* (E\pi \times_\pi (X\times X)) \] for any space $X$. 
We have that $\msc(2)$ is equivariantly homotopic to $E\pi$, so if $X$ is a $\msc$-space then there is a map 
\[ E\pi \times_\pi (X\times X) \to X \]
which induces
\[ W\tp (S_*(X) \ten S_*(X)) \to S_*(X). \]

Let $C$ be a chain complex. We define, for each $m$, a graded (non-additive) function of degree $m$
\begin{align*} q^m: C &\to W \tp (C \ten C) \\
c & \mapsto e_{m-|c|}\ten c \ten c + e_{m+1-|c|} \ten c \ten d c
\end{align*}
(interpreting terms with $e_{-n}$ as zero).
Notice that $q^m d = d q^m$ and that $q^m$ induces a homomorphism in homology.
If $C$ is a chain complex equipped with a map $W\tp (C \ten C) \to C$ (for example if $C$ is chains on a $\msc$-space) then the image of $[c]$ under the composite
\[ H_*(C) \to H_{*+m}(W\tp (C \ten C)) \to H_{*+m}(C) \] is, by definition, $Q^m [c]$. 
Since the Araki-Kudo operations $Q^m$ factor through $H_*(W\tp (C\ten C))$, we call \[ q^m: H_*(C) \to H_{*+m}(W\tp (C\ten C)) \] an `external operation'.

A naive first attempt to construct operations in our spectral sequence would involve mimicking what we did for chain complexes. We now explore this, though it won't be enough, as noted below.
The tensor product of two cosimplicial chain complexes $A_*^\bullet$ and $B_*^\bullet$ is the cosimplicial chain complex given in cosimplicial degree $p$ by $A_*^p \ten B_*^p$. If $X$ is a cosimplicial $\msc$-space, then we have a map 
\[ W\tp (S_*(X) \ten S_*(X)) \overset{E^1 \text{-iso.}}{\longrightarrow} S_*(E\pi \times_\pi (X \times X)) \to S_* X,  \] so it is useful to consider cosimplicial chain complexes $Y$ equipped with a map
\begin{equation}\label{E:c2action}
 W\tp (Y \ten Y) \to Y. \end{equation}
If $Y$ is any cosimplicial chain complex, then applying $q^m$ levelwise gives a map of cosimplicial \emph{sets}
%then we have, for each $m$, a collection of maps
%\[  q^m: Y^p \to W\tp (Y^p \ten Y^p) \] given by the formula above which constitute a 
%(degree $m$) map of cosimplicial chain complexes
\[ q^m: Y \to W\tp (Y \ten Y).\] 
One must then check (since $q^m$ is not additive) that this induces maps $E^r(Y) \to E^r(W\tp (Y \ten Y))$ for $r\geq 1$. 
Combining this with (\ref{E:c2action}) gives operations
\[ Q^m: E^r (Y) \overset{q^m}{\to} E^r(W\tp (Y\ten Y)) \to E^r(Y). \]
Notice, though, that $[y]\in E^r_{-s,t}$ is mapped to something in bidegree $(-s,t+m)$ and to zero for $m<t$, so we only pick up the vertical part of Figure~\ref{F:e2turner}. 

Henceforth, whenever we speak of external operations on a cosimplicial chain complex $Y$ we will mean operations whose target is the spectral sequence for $W\tp (Y\ten Y)$. It will grow quite tedious to write
\[ W\tp (Y \ten Y) \] for the homotopy orbit complex, so instead we will abbreviate it as
\begin{equation} \e(Y) = W\tp (Y\ten Y). \label{E:horbit}\end{equation}

\begin{rem}[See Section~\ref{S:products}] The class $[y]$ is in total degree $t-s$, so we expect there to be Araki-Kudo operations in total degrees $\geq 2(t-s)$. The vertical operations begin in total degree $2t-s$, indicating that we have missed a few.
There is one other operation that we could reasonably talk about here, the one at the bottom left. Notice that if $Y$ comes equipped with a map $W\tp (Y\ten Y) \to Y$, then there is a multiplication on the spectral sequence of $Y$. Since the bottom Araki-Kudo operation of an element is meant to be its square, it is compelling to notice that if $[y]$ is in  $E^r_{-s,t}$ then both $Q^{t-s}[y]$ and $[y]^2$ are in bidegree $(-2s,2t)$. This may convince the skeptical reader of the validity of the shape of Figure~\ref{F:e2turner}.
\end{rem}

\subsection{Bousfield-Kan Universal Examples}\label{S:introbk}

For each $p$, define $\bkspace_{(r,s,s)}^p$ as the cofiber of the inclusion
\[ \sk_{s-1} \Delta^p_+ \to \sk_{s+r-1} \Delta^p_+ \] in the category of simplicial based sets (here $\Delta_+^p$ is obtained by adding a disjoint basepoint to the standard simplicial $p$-simplex). For $t\geq s$, define $\bkspace_{(r,s,t)}^p$ by iterating the Kan suspension (\cite[2.2]{kan-ss} or \cite[III.5]{gj}) $t-s$ times.
\begin{equation} \bkspace_{(r,s,t)}^p = \Sigma^{t-s} \bkspace_{(r,s,s)}^p \label{E:boldD}\end{equation}
These cosimplicial spaces $\bkspace_{(r,s,t)}^\bullet$ were introduced in \cite{bk} where it was shown that the (integral) homology spectral sequence has the form of Figure~\ref{F:uebk}.
\begin{figure}[ht]
\centering
\begin{tikzpicture}
    %Axes
    \draw[->] (0.2,0) -- (-4,0);
    \draw[->] (0,-0.2) -- (0,3);
    %Horizontal ticks
    \draw (-1,2pt) -- (-1,-2pt) node[below] {$-s$};
    \draw (-3,2pt) -- (-3,-2pt) node[below] {$-(s+r)$};
    %Vertical Ticks
    \draw (2pt,1) -- (-2pt,1) node[right] {$t$};
    \draw (2pt,2) -- (-2pt,2) node[right] {$t+r-1$};
    %Points
    \fill (-1,1) circle (2pt) node[right] {$\imath$};
    \fill (-3,2) circle (2pt) node[left] {$\ssd^r(\imath)$};
    %Arrow
    \draw [->] (-1.2,1.1) -- (-2.8,1.9);
\end{tikzpicture}
\caption{Spectral Sequence for $\bkspace_{(r,s,t)}$}\label{F:uebk}
\end{figure}

The Bousfield-Kan example $\bkspace_{(r,s,t)}$ is weakly universal for elements in $E^r_{-s,t}$ of the homology spectral sequence. Indeed, for a cosimplicial simplicial abelian group $B$ and an element $b\in E^r_{-s,t}(B)$ there is a map of cosimplicial simplicial abelian groups $\Z \bkspace_{(r,s,t)} \to B$ which, on the spectral sequence level, sends $\imath$ to $b$. % and $\ssd^r(\imath)$ to $\ssd^r(b)$. 
We used the adjective `weakly' because the cosimplicial simplicial abelian group map actually depends on a choice of representative of $b$.
Slightly more general ideas can be found in \cite{bk}, while slightly more specific ideas can be found in section \ref{S:examplehomology}.

In any case, the spaces $\bkspace_{(r,s,t)}$ are the atomic cosimplicial spaces when it comes to the homology spectral sequence. To understand external operations, we will first understand them in these basic examples. We shall examine the spectral sequence for the cosimplicial space $E\pi \times_\pi (\bkspace_{(r,s,t)} \times \bkspace_{(r,s,t)})$.

\begin{figure}[ht]
\centering
%\begin{tikzpicture}
%    %Axes
%    \draw[->] (0.2,0) -- (-5,0);
%    \draw[->] (0,-0.2) -- (0,5);
%    %Horizontal ticks
%    \draw (-1,2pt) -- (-1,-2pt) node[below] {$-s$};
%    \draw (-4,2pt) -- (-4,-2pt) node[below] {$-2s$};
%    %Vertical Ticks
%    \draw (2pt,1) -- (-2pt,1) node[right] {$2t$};
%    %Points
%%    \fill (-4,1) circle (1.5pt);
%%    \fill (-1,1) circle (1.5pt);
%    %Arrow
%    \draw [very thick] (-4,1) -- (-1,1);
%    \draw [->,very thick] (-1,1) -- (-1,5);
%\end{tikzpicture} 
\begin{tikzpicture}
    %Axes
%    	\draw[->] (-.65,0) -- (-5,0);		
	\draw (2.2,0) -- (-5,0);		
    \draw[->] (2,-0.2) -- (2,5);
%    	\draw[->] (-.7,0) -- (-5,0);		
%	\draw (0.2,0) -- (-.3,0);		
%	\draw (-.7,0) -- (-.65,.15);
%	\draw (-.65,.15) -- (-.55, -.15);
%	\draw (-.55, -.15) -- (-.45, .15);
%	\draw (-.45, .15) -- (-.35, -.15);
%	\draw (-.35,-.15) -- (-.30, 0);
%    \draw[->] (0,-0.2) -- (0,5);
    %Horizontal ticks
    \draw (-1,2pt) -- (-1,-2pt) node[below] {$-s$};
    \draw (-4,2pt) -- (-4,-2pt) node[below] {$-2s$};
    %Vertical Ticks
    \draw (1.9,1) -- (2.1,1) node[right] {$2t$};
    %Points
%    \fill (-4,1) circle (1.5pt);
%    \fill (-1,1) circle (1.5pt);
    %Arrow
    \draw [very thick] (-4,1) -- (-1,1);
    \draw [->,very thick] (-1,1) -- (-1,5);
\end{tikzpicture} 
\caption{$E^2( E\pi \times_\pi (\bkspace_{(\infty,s,t)} \times \bkspace_{(\infty,s,t)}))$}\label{F:e2bkinf}
\end{figure}

Part of the proof of Theorem~\ref{T:turner} we present relies on a calculation (Theorem~\ref{T:e2page}) giving  Figure~\ref{F:e2bkinf} (or Figure~\ref{F:e2orig} on page \pageref{F:e2orig}), which is extremely suggestive when compared to Figure~\ref{F:e2turner} on page \pageref{F:e2turner}.
%REF 06- 

\subsection{Outline}
The next section is devoted to a careful consideration of the universal examples of Bousfield and Kan. In particular, we give a basis for the levelwise homology (Proposition~\ref{P:homologyofD}), a complete calculation of the spectral sequence (Proposition~\ref{P:univexamplecalc}), and the universal property (Proposition~\ref{P:universalprop}). The basis for the levelwise homology will be used in Section~\ref{S:e1} to give a basis for the $E^1$ page of the spectral sequence associated to $\e(D_{rst})$, which is the spectral sequence we spend much of the remainder of the paper calculating. We explore some consequences of the dual Eilenberg-Zilber theorem for cosimplicial chain complexes in Section~\ref{S:ez}, which we will need for the calculation of $E^2$ in Sections~\ref{S:isolationofrows}--\ref{S:e2}. Sections~\ref{S:einf} and \ref{S:autodiff} are dedicated to a computation of the higher differentials in this spectral sequence.

Having completed the calculation of the spectral sequence associated to $\e(D_{rst})$, we give a definition in Section~\ref{S:products} of the external operations on the level of cycles. The basic idea is that an $r$-cycle $y\in Z^r_{-s,t}(Y)$ determines a map $\dmap_y: D_{rst} \to Y$ with $\imath \mapsto [y]$. We define the operations on $[y]$ as the images of certain generators of $E^r(\e(D_{rst}))$ (the bottom right part of Figure~\ref{F:e2orig} on page~\pageref{F:e2orig}) under the map induced by $\e(\dmap_y)$. It is not at all obvious from this construction (perhaps because it is not generally true) that these operations depend only on the class $[y] \in E^r_{-s,t}(Y)$, rather than on the $r$-cycle $y$, nor is it obvious that the operations are homomorphisms.
 In Section~\ref{S:sumsofex} we show that these cycle-level operations are additive, which we use finally in Section~\ref{S:opdefinf} to show that the cycle-level operations descend to spectral sequence operations in the sense of Theorem~\ref{T:turner}. The fact that the representing maps depend on our choice of $r$-cycle introduces some indeterminacy in this step, which is precisely where the `$w$' in the theorem statement comes in. 
 On page~\pageref{exampleonlastpage} we use the universal examples to show that this `$w$' is the best possible.

%%%%%%%%%%%%%%%%%%%%%%%%%%%%%%%%%
%%%%%%%%%%%%%%%%%%%%%%%%%%%%%%%%%
\section{Algebraic Bousfield-Kan Examples}
%%%%%%%%%%%%%%%%%%%%%%%%%%%%%%%%%
%%%%%%%%%%%%%%%%%%%%%%%%%%%%%%%%%

In this section we give an explicit description of the mod-2 homology spectral sequence associated to the Bousfield-Kan universal examples. 

We give two separate constructions of the $E^1$ page of these spectral sequences. The second, starting on page~\pageref{TEXT:directproofD}, gives a complete description of the spectral sequence and allows us to establish the universal property for the Bousfield-Kan universal examples. It is also relatively quick.

The first construction is more involved and only produces $E^1$. This relies on the observation (see \cite[3.1]{bk})
that, for a cosimplicial chain complex $Y$,
\[ C(H_t(Y))^s \cong H_t(CY^s), \]
where the latter term is isomorphic to $E^1_{-s,t}(Y)$.
The method we use is to first compute the levelwise homology $H_t(Y^s)$ and then calculate the (higher) coface maps $H_t(d^1), H_t(d^2), \dots, H_t(d^{s+1})$. This will tell us about the left-hand side of the above isomorphism.

The advantage to presenting a construction along these lines is two-fold. First, it provides good practice since we will use this method of calculation later to calculate $\e(Y)$. More importantly, the levelwise homology of $Y$ gives information about the levelwise homology of $\e(Y)$ (see section~\ref{S:maincalchomology}).

Most future calculations in this paper rely on the bases we choose here.

%%%%%%%%%%%%%%%%%%%%%%%%%%%%%%%%%%%%%%%%%%%%%%%
\subsection{Homology of the Skeleton of the \texorpdfstring{$p$}{p}-Simplex}
%%%%%%%%%%%%%%%%%%%%%%%%%%%%%%%%%%%%%%%%%%%%%%%

Let $\Delta^p$\label{SYM:Delta} denote the (normalized) simplicial chains for the standard simplicial model of the $p$-simplex. A basis for $\Delta^p$ in dimension $k$ is given by the set of ordered injections $[k]\hookrightarrow [p]$.

For a complex $C$, write $\sk_t(C)$ for the brutal truncation with \[ \sk_t(C)_k = \begin{cases} C_k & k \leq t \\ 0 & k>t. \end{cases} \] This notation is chosen because of its relation to the usual notion of simplicial skeleton: if $X$ is a simplicial set, then \[ \sk_t S_* X = S_* \sk_t X, \] where $S_* = N \K$ is the normalized simplicial $\K$-chains functor. 
%REF 06-M 
Thus we wish to compute the homology of $\sk_t \Delta^p$, and since we always have \[ H_k(\sk_t C) = \begin{cases} 0 & k > t \\ Z_t (C) & k=t \\ H_k (C) & k <t\end{cases} \] %so our task is to compute the homology in dimension $t$.
we are left to understand $Z_t(\Delta^p)$. If $t>0$ then $Z_t(\Delta^p) = B_t(\Delta^p)$ and if $t=0$ then $Z_t(\Delta^p)$ is given by the collection of vertices $\K\set{[0]\hookrightarrow [p]}$.
In summary,
\[ H_k (\sk_t \Delta^p) = \begin{cases} 
B_t(\Delta^p) & k=t > 0 \\ \K\set{[0]\hookrightarrow [p]} & k=t=0 \\
\K & k=0, t> k \\
0 & \text{else.}
\end{cases} \]

We now give an explicit description for $H_t(\sk_t \Delta^p) = B_t(\Delta^p)$ when $t>0$. 

\begin{defn}
We already have $\Delta_r^p = \K\set{[r]\hookrightarrow [p]}$. Consider the set of \emph{based} injections
\[ \h^p_r = \setm{\ep}{\ep:[r]\hookrightarrow [p], \ep(0)=0}. \label{SYM:h} 
\]
%The notation is chosen because this corresponds to the $r$-simplices of the $0$-horn of the $p$-simplex.
\end{defn}

%We will use the sets $\h^p_r$ to represent a basis for various modules. Those who are already aware of the fact that $\dim_\K B_t (\Delta^p) = \binom{p}{t+1}$ (recall that we are using \emph{normalized} chains) may skip either half of the following proof since $|\h^p_{t+1}| = \binom{p}{t+1}$.

\begin{prop}\label{P:hiso}
The restriction of the differential
\[ \hd: \Delta_{t+1}^p \to \Delta_t^p \] for $t\geq 0$ 
%REF 06-O
gives an isomorphism
\[ \K \h^p_{t+1} \overset{\cong}{\to} B_t(\Delta^p). \]
\end{prop}
\begin{proof}
We define an operator $S:\Delta_t^p \to \Delta_{t+1}^p$ whose image is $\K \h^p_{t+1}$. This operator takes a basis element $\ep: [t] \hookrightarrow [p]$ to zero if $0\in \im \ep$, or, if $0\notin \im \ep$, adds zero to its image so that $\ep$ decomposes as \[ [t] \overset{d^0}\rightarrow [t+1] \overset{S(\alpha)}\longrightarrow [p].\]
Then $S^2=0$ and $\hd S + S\hd = \id_{[t]}$ for $t>0$.

To show surjectivity of $\hd: \im S \to B_t(\Delta^p)$, consider a boundary $\hd \ep$ (with $\ep: [t+1] \hookrightarrow [p]$ for some $t+1 \geq 1$). Then
\[ \hd \ep = \hd (\hd S \ep + S\hd \ep) = \hd (S\hd \ep) \]
is in the image of $\hd|_{\im S}$.

To show injectivity, suppose $S\ep: [t+1] \hookrightarrow [p]$ is in $\im S=\K \h_{t+1}^p$. If $\hd (S\ep) =0$, then
\[ S\ep = S(S\hd \ep + \hd S \ep) = S^2 \hd \ep + S(\hd (S \ep)) = 0.\]
%We first show that the map is injective. Let $V\ci \Delta_t^p$ have a basis consisting of $\ep: [t] \hookrightarrow [p]$ with $\ep(0) >0$, which is a complement of $\K \h_t^p$:  
%\[ \Delta_t^p = V \oplus \K \h_t^p. \] 
%Notice that the map \begin{align*} d_0: \K\h^p_{t+1} &\to \Delta^p_{t} \\ \ep &\mapsto \ep \circ d^0 \end{align*} is an injection. 
%The following commutes
%\[ \xymatrix{
%\K \h_{t+1}^p  \ar@{->}[r]^-{d_0} \ar@{->}[d]^\hd & V \\
%\Delta_t^p \ar@{=}[r] & V \oplus \K \h_t^p \ar@{->}[u]
%}\]
%so $\ker d \ci \ker d_0 =0$.
%
%
%To show that $\hd$ is surjective, fix a basis element $\ep: [t+1] \hookrightarrow [p]$ of $\Delta_{t+1}^p$. We need to show that $\hd\ep$ is in the image of $\hd|\K \h_{t+1}^p$. We treat the nontrivial case where $\ep(0) >0$ and write $\ep = \bar{\ep} d^0$ with $\bar{\ep} (0) = 0$. 
%Then 
%\[
%0 =\hd \hd \bar{\ep} 
%= \hd \left( \bar{\ep} \sum_{l=0}^{t+2} d^l \right) 
%= \hd (\ep) + \sum_{l=1}^{t+2} \hd ( \bar{\ep} d^l).
%\]
%Thus we see that \[ \hd (\ep) = \hd \sum_{l=1}^{t+2} \bar{\ep} \circ d^l, \] and of course $\bar{\ep} d^l (0) = 0$ for $l>0$.
\end{proof}

\begin{prop}\label{P:homtskel}
The homology of the $t$-skeleton of the standard $p$-simplex is given by
\[ H_k (\sk_t \Delta^p) \cong \begin{cases} 
\K\h_{t+1}^p & k=t > 0 \\ \K\set{[0]\hookrightarrow [p]} & k=t=0 \\
\K & k=0, t> k \\
0 & \text{else.}
\end{cases} \]
\end{prop}

%%%%%%%%%%%%%%%%%%%%%%%%%%%%%%%%%%%%%%%%%%%%%%%
\subsection{Homology Spectral Sequence of the Bousfield-Kan Examples}\label{S:examplehomology}
%%%%%%%%%%%%%%%%%%%%%%%%%%%%%%%%%%%%%%%%%%%%%%%

Fix $r\geq 1$, $s,t\geq 0$, and define $D^p=D^p_{rss}$ as the cokernel of the inclusion
\[\sk_{s-1} \Delta^p \hookrightarrow \sk_{s+r-1} \Delta^p . \label{SYM:D}\]
The cosimplicial structure of $D^\bullet$ is induced from that of $\Delta^\bullet$. 
It is not hard to see, for $t\geq s$, that  \begin{equation} D_{rst} := \Sigma^{t-s} D_{rss} \cong N \K \bkspace_{(r,s,t)}, \label{E:D}\end{equation} where $\bkspace_{(r,s,t)}$ is the cosimplicial space defined in \cite[5.1]{bk} and %section~\ref{S:introbk}
\eqref{E:boldD}, $N$ is the normalization functor $\K\operatorname{Mod}^{\Delta^{op}} \to \ch$, and $\Sigma$ is the suspension of chain complexes:
\[ (\Sigma C)_{q+1} = C_{q}. \]
%Below we make the assumption that $r\geq 2$.

\begin{prop}\label{P:homologyofD}
For $s\geq 0$ and $r\geq 2$, the homology of $D_{rss}^p$ is given by
\[ H_k(D_{rss}^p) \cong \begin{cases}
\K  \h_{s+r}^p & k=s+r-1 \\
\K \h_s^p & k=s \\
0 & \text{else.}
 \end{cases} \]
\end{prop}
\begin{proof} There is a short exact sequence of complexes
\[ 0 \to \sk_{s-1} \Delta^p \to \sk_{s+r-1} \Delta^p \to D^p \to 0 \] and, when $s\neq 1$, the result follows immediately from the associated long exact sequence and Proposition~\ref{P:homtskel}.
When $s=1$, the bottom map in the diagram
\[ \xymatrix{
0 \ar@{->}[r] & H_1 D^p_{r11} \ar@{->}[r] & H_0 \sk_0 \Delta^p \ar@{->}[r]\ar@{=}[d] & H_0 \sk_r \Delta^p \ar@{->}[r]\ar@{=}[d] & H_0 D^p_{r11} \ar@{->}[r] & 0 \\
& & \K\set{[0]\hookrightarrow [p]} \ar@{->}[r] & \K & &
}\]
is surjective and $\K \h_1^p$ is a $p$-dimensional vector space.

\end{proof}

\begin{rem}
The statement of this proposition is not true for $r<2$. We will assume that $r\geq 2$ until section~\ref{S:sumsofex}, where we will momentarily need an easy calculation for $r=1$.
\end{rem}

We now reproduce a the mod-2 version of \cite[5.3, (i)--(iii)]{bk} using Proposition~\ref{P:homologyofD}. Namely, we compute the $E^1$ page of the spectral sequence of $D_{rss}$. We do this using the fact (from \cite[3.1]{bk}) that \[ E^1_{-p,q}(D_{rss}) = H_q C(D_{rss})^p \cong C(H_q(D_{rss}))^p. \] There isn't an obvious way to obtain information about the differentials from this isomorphism, so we will not prove \cite[5.3, (iv)]{bk} using this method. The answer is given in Figure~\ref{F:uepre} (see also Figure~\ref{F:ue}).
\begin{figure}[ht]
\centering
\begin{tikzpicture}
    %Axes
    \draw[->] (0.2,0) -- (-4,0);
    \draw[->] (0,-0.2) -- (0,3);
    %Horizontal ticks
    \draw (-1,2pt) -- (-1,-2pt) node[below] {$-s$};
    \draw (-3,2pt) -- (-3,-2pt) node[below] {$-(s+r)$};
    %Vertical Ticks
    \draw (2pt,1) -- (-2pt,1) node[right] {$t$};
    \draw (2pt,2) -- (-2pt,2) node[right] {$t+r-1$};
    %Points
    \fill (-1,1) circle (2pt)  ;
    \fill (-3,2) circle (2pt) ;
    %Arrow
    %\draw [->] (-1.2,1.1) -- (-2.8,1.9);
\end{tikzpicture}
\caption{Page $1$ of $D_{rst}$}\label{F:uepre}
\end{figure}
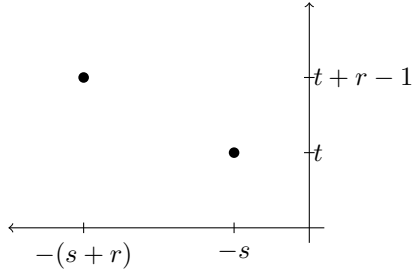

We want to compute a basis for $CHD_{rss}$ (noting that $CHD_{rst}$ is obtained from this by suspension). For $j>0$, the coface map $d^j$ takes elements of $\h_t^p$ to elements of $\h_t^{p+1}$. %(this is why we use the $0$-outer horn rather than $n$-outer horn: it is compatible with our convention for conormalization). Now if we look at the conormalization:
Applying conormalization, we find
\begin{align*} CH_{s}(D_{rss})^p &= \K \h_s^p / \left( d^1 \K \h_s^{p-1} + \dots + d^p \K \h_s^{p-1} \right) \\ &\cong \K\setm{\ep}{\ep:[s]\hookrightarrow [p], \ep(0)=0, [1,p] \ci \im \ep} \\ &=  \K\set{\id_{[s]}} \\ 
CH_{s+r-1}(D_{rss})^p & \cong \K \h_{s+r}^p /  \left( d^1 \K \h_{s+r}^{p-1} + \dots + d^p \K \h_{s+r}^{p-1} \right) \\ & \cong \K\setm{ \ep}{\ep: [s+r] \hookrightarrow [p], \ep(0) =0, [1,p] \ci \im \ep} \\ &= 
   \K \set{  \id_{[s+r]} } \\
 CH_k(D_{rss})^p &= 0 \qquad k\neq s, s+r-1. \end{align*} %REF 07-A
This is reflected in Figure~\ref{F:uepre}.
The separation of the two remaining classes means all intervening differentials $\ssd^1, \ssd^2, \dots, \ssd^{r-1}$ must be zero, so $E^1 = E^2 = \dots = E^r$.

Instead of using Proposition~\ref{P:homologyofD}\label{TEXT:directproofD}, we could conormalize $D_{rss}$ and get Figure~\ref{F:conormed}. The elements on the line with $y$-intercept 0 are $\id_{[s]}, \dots, \id_{[s+r-1]}$ and those on the line with $y$-intercept $-1$ are $\cd \id_{[s]}, \dots, \cd \id_{[s+r-1]}$. One can calculate that  \[ \cd\id_{[k]} = \hd \id_{[k+1]} \] in $C\Delta^\bullet$ since $\hd \id_{[k+1]} = \id_{[k+1]}d^0 = d^0 = \cd = \cd \id_{[k]}$.

\begin{figure}[ht]
\centering
\begin{tikzpicture}
    %Axes
    \draw[->] (0.2,0) -- (-4,0);
    \draw[->] (0,-0.2) -- (0,4);
    %Horizontal ticks
    \draw (-.5,2pt) -- (-.5,-2pt) node[below] {$-s$};
    \draw (-1,2pt) -- (-1,-2pt); % node[below] {$-(s+r)$};
    \draw (-2.5,2pt) -- (-2.5,-2pt); 
    \draw (-3,2pt) -- (-3,-2pt); %node[below] {$-v$}; 
    %Vertical Ticks
    \draw (2pt,.5) -- (-2pt,.5) node[right] {$s$};
    \draw (2pt,1) -- (-2pt,1) node[right] {$s+1$};
    \draw (2pt,2.5) -- (-2pt,2.5) node[right] {$s+r-2$};
    \draw (2pt,3) -- (-2pt,3) node[right] {$s+r-1$};
    %Points on x=-y
    \fill (-.5,.5) circle (2pt);
    \fill (-1,1) circle (2pt);
    \fill (-2.5,2.5) circle (2pt);
    \fill (-3,3) circle (2pt);
    %Points on x=-y-1
    \fill (-1,.5) circle (2pt);
    \fill (-1.5, 1) circle (2pt);
    \fill (-2.5,2) circle (2pt);
    \fill (-3,2.5) circle (2pt);
    \fill (-3.5,3) circle (2pt);
    %Horizontal Arrows
    \draw [->] (-.6,.5) -- (-.9,.5);
    \draw [->] (-1.1,1) -- (-1.4,1);
    \draw [->] (-2.1,2) -- (-2.4,2);
    \draw [->] (-2.6,2.5) -- (-2.9,2.5);
    \draw [->] (-3.1,3) -- (-3.4,3);    
    %Vertical Arrows
    \draw [->] (-1,.9) -- (-1,.6);
    \draw [->] (-1.5,1.4) -- (-1.5,1.1);
    \draw [->] (-2.5,2.4) -- (-2.5,2.1);
    \draw [->] (-3,2.9) -- (-3,2.6);
    %Dotted Lines
    \draw [dotted] (-1.5,1.5)--(-2,2);
%    \draw [dotted] (-1.5,1)--(-2.5,2);
    \draw [dotted] (-1.8,1.3)--(-2.2,1.7);
\end{tikzpicture}
\caption{The Bicomplex $C(D_{rss})$}\label{F:conormed}
\end{figure}

\begin{prop}\label{P:univexamplecalc} If $2\leq r \leq \infty$, then
\[ E^1_{p,q}(D_{rst}) = E^r_{p,q}(D_{rst}) = \begin{cases} \K & (p,q) = (-s,t) \text{ or } (-s-r, t+r-1) \\ 0 & \text{otherwise.}\end{cases} \]
If, furthermore, $r < \infty$, then $E^{r+1}(D_{rst}) = 0$. Finally, 
\[ H_k(\Tot CD_{rst}) = \begin{cases} \K & k = t-s \text{ and } r=\infty \\
0 & \text{otherwise,} \end{cases} \]
where $\Tot$ is the product totalization \eqref{E:bitot}.
\end{prop}
\begin{proof}
This essentially follows from Figure~\ref{F:conormed} in the case $s=t$ and by iterated suspension otherwise.
There is exactly one $r$-cycle in total degree $0$, namely  \begin{equation} \imath = \sum_{k=s}^{s+r-1} \id_{[k]}. \label{E:iota}\end{equation} Applying the total differential, we have $\partial \imath = \cd \id_{[s+r-1]}$ (or $0$ if $r=\infty$). 
\end{proof}

This proof tells us why $E^{r+1} D_{rss} = 0$ for $r<\infty$, which we did not show using the first construction. Thus, the spectral sequence is as in Figure~\ref{F:ue}.

\begin{figure}[ht]
\centering
\begin{tikzpicture}
    %Axes
    \draw[->] (0.2,0) -- (-4,0);
    \draw[->] (0,-0.2) -- (0,3);
    %Horizontal ticks
    \draw (-1,2pt) -- (-1,-2pt) node[below] {$-s$};
    \draw (-3,2pt) -- (-3,-2pt) node[below] {$-(s+r)$};
    %Vertical Ticks
    \draw (2pt,1) -- (-2pt,1) node[right] {$t$};
    \draw (2pt,2) -- (-2pt,2) node[right] {$t+r-1$};
    %Points
    \fill (-1,1) circle (2pt) node[right] {$\imath$};
    \fill (-3,2) circle (2pt) node[left] {$\ssd^r(\imath)$};
    %Arrow
    \draw [->] (-1.2,1.1) -- (-2.8,1.9);
\end{tikzpicture}
\caption{Pages $2$ through $r$ of $D_{rst}$}\label{F:ue}
\end{figure}

\begin{prop}[Universal Property]\label{P:universalprop}
Let $Y$ be a cosimplicial chain complex and $y\in Z_{-s,t}^r(Y)$. Then there is a map \[ \dmap_y : D_{rst} \to Y \] with \begin{align*} E^r (\dmap_y) (\imath) = [y] && E^r(\dmap_y) (\cd \id_{[s+r-1]}) = \ssd^r [y]. \end{align*}
\end{prop}
We will be using the definition of $\dmap_y$ frequently.
\begin{defn}[Representing Map]
Let \[ y\in Z_{-s,t}^r (Y) \ci F^{-s} \Tot C(Y)_{t-s} = \prod_{p\geq s} CY^p_{t-s+p} \] which we write as
\[ y= \sum_{k=0}^\infty y_{t+k}^{s+k} \qquad \qquad y_q^p \in C(Y)_q^p. \]
Define $C(\dmap_y)$ by
\begin{align*}
\Sigma^{t-s} \id_{[s+k]} & \mapsto y_{t+k}^{s+k} \\
\Sigma^{t-s} \cd \id_{[s+k]} & \mapsto \cd y^{s+k}_{t+k}.
\end{align*} The following proof will show that this is a map of bicomplexes, so $C(\dmap_y)$ gives $\dmap_y$ by the Dold-Kan theorem.
\end{defn}
\begin{proof}[Proof of Proposition~\ref{P:universalprop}]
Since $\partial y \in F^{-s-r}$ we know that
\[ \cd y_{t+k}^{s+k} = \hd y_{t+k+1}^{s+k+1} \] for $0\leq k \leq r-2$, which shows that $C(\dmap_y)$ is a map of bicomplexes. Furthermore, using the notation from page~\pageref{NOTATION:SS}, we have
\[ C(\dmap_y) \imath = \sum_{k=0}^{r-1} y_{t+k}^{s+k} \sim_r \sum_{k=0}^\infty y_{t+k}^{s+k} = y \] and
\[ C(\dmap_y) \partial \imath = \cd y_{t+r-1}^{s+r-1}.\] 
\end{proof}

%%%%%%%%%%%%%%%%%%%%%%%%%%%%%%%%%
%%%%%%%%%%%%%%%%%%%%%%%%%%%%%%%%%

%%%%%%%%%%%%%%%%%%%%%%%%%%%%%%%%%%%%%%%%%%%%%%%
\section{The \texorpdfstring{$E^1$}{E-1} Page}\label{S:e1}
%%%%%%%%%%%%%%%%%%%%%%%%%%%%%%%%%%%%%%%%%%%%%%%

We are interested in the spectral sequence associated to the cosimplicial chain complex
$ \e(D_{rst})$ (where $\e$ is a homotopy orbit complex, see \eqref{E:horbit}).  Note that $\e(D_{rst}) \cong \Sigma^{2t-2s} \e(D_{rss})$ where the suspension is taken levelwise, so it is enough to understand the spectral sequence for $\e(D_{rss})$. Let $Y$ be a cosimplicial chain complex; following \cite[3.1]{bk},
$E^1_{-p}(Y)$  is isomorphic to  $C(H_*Y)^p$. We note that the isomorphism $CH_*(-) \cong H_*C(-)$ is an isomorphism of complexes of graded modules. It is not true in general that the spectral sequence associated to $Y$ is isomorphic to the spectral sequence associated to $H_*(Y)$.

Fix $r$ and $s$; we now turn to computing $C(H_*\e(D_{rss}))$. Below we often use the abbreviation
$ D^p = D^p_{rss}$.
%%%%%%%%%%%%%%%%%%%%%%%%%%%%%%%%%%%%%%%%%%%%%%%
\subsection{Homology}\label{S:maincalchomology}
%%%%%%%%%%%%%%%%%%%%%%%%%%%%%%%%%%%%%%%%%%%%%%%

To compute $H_*(\e(D_{rss}))$, first notice that, for a map of chain complexes $C \to C'$ and choices of quasi-isomorphisms $H(C) \to C$ and $H(C') \to C'$, the diagram
\[ \xymatrix{ 
H(C) \ar@{->}[r] \ar@{->}[d] & C \ar@{->}[d] \\
H(C') \ar@{->}[r] & C'
}\]
commutes up to homotopy. It follows from \cite[Lemma 1.1]{may} that we have a \emph{natural transformation}
\[ H_* (\e( H_*(-))) \to H_*(\e(-)) \]
which is an isomorphism for each object.
Thus
 \[ H_* \Big( \e(D_{rss}^p)\Big) \cong H_*\Big( \e( H_*(D_{rss}^p)) \Big) \] and this is an isomorphism of cosimplicial graded modules.

Once we have made this change, notice that, for a $\K\pi$-module $M$ (such as $H_*(D^p) \ten H_*(D^p)$), the complex $W\tp M$ is just 
\[ \cdots M \overset{1+\sigma}{\longrightarrow} M \overset{1+\sigma}{\longrightarrow} M \overset{1+\sigma}{\longrightarrow} M \to 0 .\] Thus the homology is $M/(1+\sigma)$ in the bottom dimension and $\ker(1+\sigma) / \im(1+\sigma)$ in dimensions bigger than zero. 
This essentially gives the following special case of \cite[Lemma 1.3]{may}:
\begin{lem}\label{L:mayonethree}
Let $K$ be a $\K$-module with totally ordered basis $\setm{x_j}{j\in J}$. Let $A\ci K \ten K$ have basis $\setm{x_j \ten x_j}{j\in J}$ and $B\ci K\ten K$ have basis $\setm{x_{j_1} \ten x_{j_2}}{j_1 < j_2, \text{where }j_1, j_2 \in J}$. Then
\[ H(W\tp (K\ten K)) \cong \left( \bigoplus_{i=0}^\infty e_i \ten A \right) \oplus (e_0 \ten B).\]
\end{lem}

%If we are working with $\sk_n W \tp M$ instead, then we also get top dimensional homology $\ker(1+\sigma)$. 
%With a little more work this gives \cite[Lemma 1.3]{may}, which we now use.

We choose an order for the basis $\h_s^p \sqcup \h_{s+r}^p$ of $H_*(D_{rss}^p)$ that we found in 
Proposition~\ref{P:homologyofD}.
%\S\ref{S:examplehomology}.

\begin{defn}[Total Order]
%Every injection $\ep: [m]\hookrightarrow [p]$ is given by a unique decreasing sequence $r_1 > r_2 > \dots > r_{p-m}$ (namely, the complement of the image) with $\ep=d^{r_1} \dots d^{r_{p-m}}$. For a fixed $m$ we declare the order on injections to be given by the reverse dictionary order on their associated sequence. For example, $d^3d^0 < d^2d^1 < d^2d^0$ and $d^3d^2 < d^3 d^1$. By Proposition~\ref{P:homologyofD}, we then have an induced order on $H_{s+r-1}(D^p)$ and on $H_{s}(D^p)$. We give $H_*(D^p)$ an order by declaring that $H_{s}(D^p) < H_{s+r-1}(D^p)$.
We associate to $\ep: [m]\hookrightarrow [p]$ the word of length $p+1$, whose $i^{\text{th}}$ letter is $0$ if $i\notin \im \ep$ and $1$ if $i\in \im \ep$. 
For a fixed $m$ we declare the order on injections to be given by the reverse lexicographic order on their associated words. %For example, $d^3d^0 < d^2d^1 < d^2d^0$ and $d^3d^2 < d^3 d^1$. 
By Proposition~\ref{P:homologyofD}, we then have an induced order on $H_{s+r-1}(D^p)$ and on $H_{s}(D^p)$. We give $H_*(D^p)$ an order by declaring that $H_{s}(D^p) < H_{s+r-1}(D^p)$.
\end{defn}

We apply Lemma~\ref{L:mayonethree} with this totally ordered basis to see that the homology of $\e(H_*(D^p))$ has a basis given by the disjoint union of the following sets:
\begin{equation}\label{E:sets}
\begin{gathered} 
\setm{e_m\tensor \ep \tensor \ep}{\ep\in \h_s^p, m\in \N} \\
\setm{e_m \tensor \gamma \tensor  \gamma}{\gamma\in \h_{s+r}^p, m\in \N} \\
\setm{e_0 \tensor \ep \tensor \ep'}{\ep, \ep'\in \h_s^p, \ep < \ep'}\\
\setm{e_0 \tensor \ep \tensor \gamma}{\ep\in \h_s^p, \gamma\in \h_{s+r}^p}\\
\setm{e_0 \tensor \gamma \tensor \gamma'}{\gamma, \gamma'\in\h_{s+r}^p, \gamma < \gamma'}
\end{gathered}
\end{equation}
Each $\gamma$ appearing in the above tensor products stands for $\hd \gamma$ as in the isomorphism in Proposition~\ref{P:homologyofD}. %Proposition~\ref{P:hiso}. 
%
%To understand what happens when we apply $d^k$ to one of these basis elements, we must use the isomorphism from Proposition~\ref{P:homologyofD}. 
%Notice that for $\gamma \in \h_{s+r}^p$, t
This isomorphism is induced from $\hd$ (see Proposition~\ref{P:hiso}) and, for $k>0$, $d^k \hd \gamma = \hd d^k \gamma$. Thus we may use $d^k: \K\h_{s+r}^p \to \K\h_{s+r}^{p+1}$ for $k>0$ without worry.

\begin{lem}
For $k> 0$, the coface maps \[
d^k:\e(H_*(D^p)) \to \e(H_*(D^{p+1})) \] respect the basis for $H_*(\e(H_*(D^\bullet)))$ given above (\ref{E:sets}). 
\end{lem}
\begin{proof}
In the word associated to an injection, the operator $d^k$ inserts a $0$ in the $k^{\text{th}}$ position, an operation which preserves the reverse lexicographic order.
%One must check that $d^k \ep < d^k \ep'$ if $\ep < \ep'$. %But this is easy to see. 
%We may reduce to the case where $\ep=d^{r_1} \dots d^{r_t}$ and $\ep'=d^{v_1}\dots d^{v_t}$ with $r_1 > v_1$. We begin to rewrite $d^k \ep$ and $d^k \ep'$ in the canonical form. If $k> r_1$ then $d^k d^{r_1} \cdots$ and $d^k d^{v_1} \cdots$ are already in the canonical form, and the order is preserved. If $k\in (v_1, r_1]$ then $d^k \ep = d^{r_1+1}d^k \cdots$ and $d^k\ep' = d^k d^{v_1} \cdots$ but of course $r_1+1 > k$ so the order is preserved. If $k\leq v_1$ then $d^k \ep = d^{r_1+1} \cdots$ and $ d^{k}\ep' = d^{v_1+1} \cdots$ but of course $r_1+1 > v_1+1$.
\end{proof}

%%%%%%%%%%%%%%%%%%%%%%%%%%%%%%%%%%%%%%%%%%%%%%%
\subsection{Conormalization}\label{SS:conorm}
%%%%%%%%%%%%%%%%%%%%%%%%%%%%%%%%%%%%%%%%%%%%%%%
We just saw that $d^k$, $k>0$, sends basis elements in $H_*(\e(H_*(D^p)))$ to basis elements in $H_*(\e(H_*(D^{p+1})))$ via \begin{gather*} e_m \ten \ep \ten \ep' \mapsto e_m \ten d^k \ep \ten d^k \ep' \\ \ep, \ep' \in \h_{s}^p \cup \h_{s+r}^p \qquad m\geq 0. \end{gather*}
Thus the conormalization (see \eqref{E:conorm})
\[ CH_*\Big( \e( H_*(D^{p}))\Big) \] has a basis consisting of elements of the original basis which are not in the image of $d^k$ for $k=1,\dots, p$. 

\begin{thm}\label{T:e1basis} Let $2\leq r \leq \infty$ and $s\geq 0$. The $E^1$ page of the spectral sequence for the cosimplicial chain complex $\e(D_{rst})$ can be given a basis consisting of the following:

If $r=\infty$, then the basis consists of the disjoint union of the following two sets:
\begin{equation}\label{E:basisrinf}
\begin{aligned} &\setm{e_m \tensor \id_{[s]} \tensor \id_{[s]} \in E^1_{-s,2t+m}}{m\geq 0} \\
 &\setm{e_0 \tensor \ep \tensor \ep' \in E^1_{-p,2t}}{  \ep < \ep' \in \h_s^p \text{ and } [p]=\im \ep \cup \im \ep'}. \end{aligned} \end{equation} 
If $r< \infty$, then the basis is given by the disjoint union of the sets (\ref{E:basisrinf}) with the sets
\begin{align*}
&\setm{ e_m \tensor  \id_{[s+r]} \tensor  \id_{[s+r]} \in E^1_{-s-r,2t+2r+m-2}}{m\geq 0} \\
&\setm{ e_0 \tensor \ep \tensor \gamma \in E^1_{-p,2t+r-1}}{ \ep\in \h_s^p, \gamma\in \h_{s+r}^p \text{ and }[p]=\im \ep \cup \im \gamma} \\
&\setm{ e_0 \tensor  \gamma \tensor  \gamma' \in E^1_{-p,2t+2r-2}}{\gamma <\gamma' \in \h_{s+r}^p \text{ and } [p]=\im \gamma \cup \im \gamma'}.
\end{align*}

\end{thm}
\begin{proof}
In the spectral sequence associated to $\e(D_{rss})$, $E^1_{-p}\cong CH_*\Big( \e( H_*(D^{p}))\Big)$ has a basis consisting of elements of the original basis which are not in the image of $d^k$ for $k=1,\dots, p$. This basis is given by the disjoint union of the following sets:
\begin{gather*} 
\setm{e_m\tensor \ep \tensor \ep}{\ep\in \h_s^p, [1,p] \ci \im \ep, m\in \N} \\
\setm{e_m \tensor \gamma \tensor  \gamma}{\gamma\in \h_{s+r}^p,[1,p] \ci \im\gamma, m\in \N} \\
\setm{e_0 \tensor \ep \tensor \ep'}{\ep, \ep'\in \h_s^p,[1,p]\ci \im \ep \cup \im \ep', \ep < \ep'}\\
\setm{e_0 \tensor \ep \tensor  \gamma}{\ep\in \h_s^p, \gamma\in \h_{s+r}^p, [1,p]\ci \im \ep \cup \im \gamma}\\
\setm{e_0 \tensor  \gamma \tensor \gamma'}{\gamma, \gamma'\in \h_{s+r}^p,[1,p]\ci \im \gamma \cup \im \gamma' ,  \gamma < \gamma'}.
\end{gather*} Furthermore, recall that $\ep$ and $\gamma$ are injective maps and $0$ is in the image of each. This establishes the $s=t$ case of the theorem. Now apply the $(2t-2s)$-fold suspension in the vertical direction to this basis.
\end{proof}

A picture of the $E^1$ page is given in Figure~\ref{F:e1orig}, where we have indicated modules with rank greater than zero by snaky lines and modules of rank one with straight lines.
The reader is encouraged to compare this to the picture of $E^2$ given in  Figure~\ref{F:e2orig} on page~\pageref{F:e2orig}.

\begin{figure}[ht]
\centering
%\scalebox{.7} 
{
\begin{tikzpicture}
    %Axes
%    \draw[->] (0.2,0) -- (-8,0);
	\draw[->] (-.65,0) -- (-8,0);		%new
	\draw (0.2,0) -- (-.35,0);		%new
	\draw (-.75,-.1) -- (-.55,.1);	%new
	\draw (-.45,-.1) -- (-.25,.1);	%new
%	\draw [zigzag] (-0.2,0) -- (-.8,0);
%	\draw[->]  (-.2,0) -- (-.8.0);
    \draw[->] (0,-0.2) -- (0,5);
    %Horizontal ticks
    \draw (-1,2pt) -- (-1,-2pt) node[below=.01] {$-s$};
    \draw (-2.5,2pt) -- (-2.5,-2pt) node[below=.01] {$-s-r$};
    \draw (-4,2pt) -- (-4,-2pt) node[below=.01] {$-2s$};
    \draw (-5.5,2pt) -- (-5.5,-2pt) node[below=.01] {$-2s-r$};
    \draw (-7,2pt) -- (-7,-2pt) node[below left=.01 and -.6] {$-2s-2r$};
    %Vertical Ticks
    \draw (2pt,1) -- (-2pt,1) node[right] {$2t$};
    \draw (2pt,2) -- (-2pt,2) node[right] {$2t+r-1$};
    \draw (2pt,3) -- (-2pt,3) node[right] {$2t+2r-2$};
    %Horizontal Lines
    \draw [snake=snake,very thick] (-5.5,2) -- (-2.5,2);
    %\draw [snake=coil,segment aspect=0,very thick] (-5.5,2) -- (-2.5,2);
    %Arrow
    \draw [->, snake=snake,very thick] (-4,1) -- (-1,1) [snake=none] -- (-1,5);
    %\draw [->,very thick] (-4,1) -- (-1,1) -- (-1,5);
    \draw [->, snake=snake,very thick] (-7,3) -- (-2.5,3) [snake=none]--  (-2.5,5);
\end{tikzpicture} 
}
\caption{$E^1(  \e(D_{rst}))$}\label{F:e1orig}
\end{figure}

\begin{rem}
Using the notation of Theorem~\ref{T:e1basis}, we note that that the basis elements $e_0 \ten \ep \ten \ep'$ live in cosimplicial degrees between $-s-1$ and $-2s$, $e_0 \ten \ep \ten \gamma$ live in cosimplicial degrees between $-s-r$ and $-2s-r$, and $e_0\ten \gamma \ten \gamma'$ live in cosimplicial degrees between $-s-r-1$ and $-2s-2r$.
\end{rem}

It is possible to give an explicit description of the differential $\ssd^1$ in terms of this basis. The interested reader may find formulas in \cite{mythesis}, but they are not needed for the calculation of $E^2$.

%%%%%%%%%%%%%%%%%%%%%%%%%%%%%%%%%%%%%%%%%%%%%%%
\section{Spectral Sequence of \texorpdfstring{$X\ten Y$}{X x Y}}\label{S:ez}
%%%%%%%%%%%%%%%%%%%%%%%%%%%%%%%%%%%%%%%%%%%%%%%
In this section we examine the spectral sequence associated to the tensor product of two cosimplicial chain complexes, anticipating applications for computing $E^2(\e(D_{rst}))$ (in section~\ref{S:isolationofrows}) and for the external product (in section~\ref{S:products}).

Let $X$ and $Y$ be cosimplicial chain complexes. 
There are two bicomplexes, $C(X) \ten C(Y)$ and $C(X\ten Y)$, which are readily associated to the pair. We now give natural transformations
\[ C(X) \ten C(Y) \rightleftarrows C(X\ten Y), \] 
which are the cosimplicial versions of the Alexander-Whitney and shuffle maps on normalized chains that arise when considering simplicial abelian groups, as in \cite[Corollaries VIII.8.6 and VIII.8.9]{homology}.
\begin{defn}[{\cite[Appendix]{bauesmuro}}]\label{D:awshuff}
The Alexander-Whitney map $AW$ is defined on $C(X)^p \ten C(Y)^q$ by
\begin{equation} AW(x^p\ten y^q) = d^{p+q} \cdots d^{p+1} x \ten d^{p-1} \cdots d^{0} y. \label{E:AW}\end{equation} The shuffle map $\nabla$ is defined on $C(X\ten Y)^n$ by
\[ \nabla(x^n \ten y^n) = \sum_{p+q=n} \sum_{\substack{ (p,q)-\text{shuffles}\\ \tau}} s^{\tau(p)} \cdots s^{\tau(p+q-1)} x \ten s^{\tau(0)} \cdots s^{\tau(p-1)} y \]
where we consider $(p,q)$-shuffles as permutations of the set \[ \set{0,1,\dots, p+q-1}. \]
\end{defn}

\begin{lem} The Alexander-Whitney map and shuffle map are maps of bicomplexes.
\end{lem}
\begin{proof} Regarding $X$ and $Y$ simply as cosimplicial modules, the usual statement is that these are both maps in $\ch^{\geq 0} (\K\operatorname{Mod})$. Since the cosimplicial operators commute with the internal differential $X^p_i \to X^p_{i-1}$, these are actually both maps in $\ch^{\geq 0}(\ch)$. Finally, then tensor product on $\ch^{\geq 0}(\ch)$ is given by \[ (A\ten B)^p_q = \oplus_{i+j=p} A^i \ten B^j = \oplus_{i+j=p,m+n=q} A^i_m \ten B^j_n\] which is compatible with the isomorphism between the categories of bicomplexes and cochain complexes in complexes. \end{proof}

Notice that  $C(X)\ten C(Y)$ is a retraction of $C(X\ten Y)$:
\[ \nabla \circ AW = \id_{C(X)\ten C(Y)}. \]
Furthermore, if $X$ and $Y$ are cosimplicial abelian groups, the dual Eilenberg-Zilber theorem (see, for example, the appendix in \cite{bauesmuro}) tells us that $\nabla$ and $AW$ are inverse chain homotopy equivalences. In the case when $X$ and $Y$ are cosimplicial chain complexes, we can extend this to show that these maps give isomorphisms on $E^2$ (Proposition~\ref{P:easye2}).

We begin with a small result on the spectral sequence associated to the tensor product of bicomplexes. 
\begin{lem}\label{L:tensbicplx}
Let $B$ and $B'$ be bicomplexes (over $\K$) and consider the spectral sequences obtained by filtering by columns. Then the map
\[ E^r(B) \ten E^r(B') \to E^r(B\ten B') \]
is an isomorphism.
\end{lem}
We will generally identify these two bigraded modules.
\begin{proof}[Proof of Lemma~\ref{L:tensbicplx}] The tensor product 
\[ (B\ten B')_{p,q} = \bigoplus_{i,j} B_{i,j}\ten B'_{p-i,q-j} \] is again a bicomplex. The result follows by iterated application of the K\"unneth isomorphism. Specifically, we use the induction and have the isomorphism
%\[ E^r B \ten E^r B' = H(E^{r-1}B)\ten H(E^{r-1}B') \overset{\cong}{\to} H(E^{r-1}(B\ten B')) = E^r(B\ten B') \] with the base case coming from $E^0 B \ten E^0 B' = B\ten B' = E^0 (B\ten B')$. 
\begin{align*}
(E^rB\ten E^rB')_{p,q} &= \bigoplus_{i,j} E^r_{i,j}(B) \ten E^r_{p-i,q-j}(B') \\
&= \bigoplus_{i,j} H(E^{r-1}(B), \ssd^{r-1})_{i,j} \ten H(E^{r-1}(B'), \ssd^{r-1})_{p-i,q-j} \\
&\cong H\left( (E^{r-1}(B), \ssd^{r-1}) \ten (E^{r-1}(B'), \ssd^{r-1})\right)_{p,q} \\
&\overset\bigstar\cong H(E^{r-1}(B\ten B'), \ssd^{r-1})_{p,q} \\
&=E^r_{p,q}(B\ten B')
\end{align*}
where $\bigstar$ is the statement for $r-1$ and the base case is 
\begin{multline*} (E^0 B \ten E^0 B')_{p,q} = \bigoplus_{i,j} E^0_{i,j} (B) \otimes E^0_{p-i,q-j}(B')  \\ \cong \bigoplus_{i,j} B_{i,j} \otimes B'_{p-i,q-j} = (B\ten B')_{p,q} \cong E^0_{p,q} (B\ten B').\end{multline*}
\end{proof}

%We will generally wish to identify 
%\[ E^r B \ten E^r B' = E^r(B\ten B'). \]

%\newcommand{\newaw}{\mathbf{AW}}

\begin{prop}\label{P:easye2}
Let $X$ and $Y$ be cosimplicial chain complexes. The Alexander-Whitney map
\[ C(X)\ten C(Y) \to C(X\ten Y) \] induces an isomorphism 
\[ E^r(X) \ten E^r(Y) \overset{\cong}{\to} E^r(X\ten Y) \] for all $r\geq 2$. The inverse is induced from the shuffle map $\nabla$.
\end{prop}
\begin{proof} By Lemma~\ref{L:tensbicplx} it is enough to show that the map of bicomplexes $C(X) \ten C(Y) \to C(X\ten Y)$ induces a an isomorphism on page 2 of the associated spectral sequence:
\[ E^2 (C(X) \ten C(Y)) \overset{\cong}{\to} E^2(C(X\ten Y)). \]
Consider the diagram
\[ \xymatrix@R=.8pc{
CH_*(X) \ten CH_*(Y) \ar@{->}[r]^{AW} \ar@{=}[d] & C[H_*(X) \ten H_*(Y) ] \ar@{->}[r]^-\cong & CH_*(X\ten Y) \ar@{=}[d] \\
H_*C(X) \ten H_*C(Y) \ar@{->}[r]^-\cong & H_*[C(X) \ten C(Y) ] \ar@{->}[r]^-{H_*AW} & H_*C(X\ten Y)
}\]
where the isomorphisms come from the K\"unneth theorem. It is easy to see that this commutes when we consider $CH_*(X) = H_*C(X)$ as a subobject of $H_*(X)$. 
The dual Eilenberg-Zilber theorem implies that the top left map $AW$ becomes an isomorphism when we take homology in the horizontal direction. The same thus applies to  $H_*AW$, implying that the composite
\[ E^2(X) \ten E^2(Y) \overset{\cong}{\to} E^2(C(X) \ten C(Y)) \overset{E^2(AW)}{\longrightarrow} E^2(X\ten Y) \] is an isomorphism. 

Since $\nabla AW = \id$, the inverse map must be the one induced from $\nabla$.
\end{proof}

\begin{rem}\label{R:commutative}
One consequence of this proposition is that although \[ AW: C(Y)\ten C(Y) \to C(Y\ten Y)\] is not $\pi$-equivariant, it becomes so on $E^2$ (see \cite[Theorem 9.3(vii)]{bk}). This is because at the level of bicomplexes $\nabla$ is $\pi$-equivariant:
\[ AW \sigma = AW \sigma \id = AW \sigma (\nabla AW )  = AW (\nabla \sigma) AW \sim_{2} \id \sigma AW = \sigma AW, \]
where $\sim_2$ indicates that  these induce the same map on $E^2$.\end{rem}

%%%%%%%%%%%%%%%%%%%%%%%%%%%%%%%%%%%%%%%%%%%%%%%
\section{Isolation of the Rows}\label{S:isolationofrows}
%%%%%%%%%%%%%%%%%%%%%%%%%%%%%%%%%%%%%%%%%%%%%%%

Recall from Figure~\ref{F:e1orig} that the $E^1$ page of the spectral sequence for $\e(D_{rst})$ contains three `horizontal strips' $[-2s,-s] \times \set{2t}$, $[-2s-r,-s-r]\times \set{2t+r-1}$, and $[-2s-2r,-s-r] \times \set{2t+2r-2}$ (if $r=\infty$ we only have the first of these) which are the only places where $\ssd^1$ may be nonzero. 
We introduce a slightly more general class of complexes in this section (basically including the $r=0$ and $r=1$ cases of the middle horizontal strip) to facilitate this computation, and quickly compute the cohomology of the middle horizontal strip. In section~\ref{S:cohomvom} we will compute the cohomology of the top and bottom strips.

Fix $s$ and $s'$ nonnegative integers and let $\voc_{s,s'}$ be the cochain complex
\begin{equation} \voc_{s,s'} = C(H_s(D_{\infty ss}) \ten H_{s'}(D_{\infty s's'})). \label{E:voc} \end{equation} When $s=s'$, $\voc_{s,s}$ has an obvious $\pi$-action and we define 
\begin{equation} \vom_s = \voc_{s,s} / \pi. \label{E:vom} \end{equation} We know from Proposition~\ref{P:homologyofD} that a basis for $H_s(D_{\infty ss})$ is given by \[ \h_s^p = \setm{\zeta}{\zeta:[s]\hookrightarrow [p], \zeta(0)=0}, \] and we take $\h_s^p \times \h_{s'}^p$ as the preferred basis for $H_s(D_{\infty ss}^p) \ten H_{s'}(D_{\infty s's'}^p)$.
Let \[ \vo^p_{s,s'} \ci \h_s^p \times \h_{s'}^p \label{SYM:vo}\]  be the set of pairs $(\zeta, \zeta')$ with $[p] = \im \zeta \cup \im \zeta'$.

\begin{lem}
The set $\vo^p_{s,s'}$ is a basis for $\voc_{s,s'}^p$. 
\end{lem}
\begin{proof} 
Recall that 
% is a basis for $H_s(D_{\infty ss})$ by Proposition~\ref{P:homologyofD}. 
%We take $\h_s^p \times \h_{s'}^p$ as the preferred basis for $H_s(D_{\infty ss}^p) \ten H_{s'}(D_{\infty s's'}^p)$. 
if $\zeta \in \h_s^p$, $\zeta' \in  \h_{s'}^p$, and $k>0$, then  $d^k \zeta \in \h_s^{p+1}$ and $d^k \zeta' \in \h_{s'}^{p+1}$.  Thus $d^k$ takes basis elements to basis elements for $k > 0$.  
A basis element $(\zeta, \zeta') \in \h_s^p \times \h_{s'}^p$ with $[1,p] \ci \im \zeta \cup \im \zeta'$ is not in $\im d^k$ for any $k>0$, hence is not in $\im d^1 + \dots + \im d^p$. All other basis elements are in $\im d^k$ for some $k>0$.  Thus the set of elements $(\zeta, \zeta')$ with $[1,p] \in \im \zeta \cup \im \zeta'$ constitutes a basis for a complement of $\im d^1 + \cdots + \im d^p$. The result follows from the definition of conormalization.
\end{proof}

\begin{rem} Observe that $\vo^p_{s,s'}$ is nonempty exactly when $p\in [ \max(s,s'), s+s']$.
\end{rem}

\begin{prop}\label{P:comparison}
Fix $r,s,t$, and consider the spectral sequence for $\e( D_{rst})$. There are maps
\begin{align*}
\psi_{bot}:& \vom_{s}^p  \to E^1_{-p,2t} \\
\psi_{mid}:& \voc_{s,s+r}^p  \to E^1_{-p,2t+r-1} \\
\psi_{top}:& \vom_{s+r}^p  \to E^1_{-p,2t+2r-2}
\end{align*}
for $p\geq 0$ which are isomorphisms of complexes.
\end{prop}
\begin{proof}
Assume $t=s$. We have isomorphisms of cosimplicial modules
\begin{align*}
(H_s(D_{rss}) \ten H_s(D_{rss}) ) / \pi & \to H_{2s} (\e(H_*(D_{rss}))) \\
H_s(D_{rss}) \ten H_{s+r-1} (D_{rss}) &\to H_{2s+r-1} (\e(H_*(D_{rss}))) \\
(H_{s+r-1}(D_{rss}) \ten H_{s+r-1}(D_{rss}) ) / \pi & \to H_{2s+2r-2} (\e(H_*(D_{rss})))
\end{align*}
each given by $\zeta \ten \zeta' \mapsto e_0 \ten \zeta \ten \zeta'$. Applying $C$ to the modules on the right gives the nontrivial rows of $E^1$.

We now identify the left hand side in the above isomorphisms. The inclusion $D_{rss} \to D_{\infty ss}$ induces an isomorphism $H_s(D_{rss}) \overset{\cong}{\to} H_s(D_{\infty ss})$. 
Since $H_{s+r} \Delta^p = 0 = H_{s+r-1} \Delta^p$ and $\sk_{s-1} \Delta^p$ is zero in degree $s+r-1$ and degree $s+r-2$, the long exact sequences associated to the short exact sequences
\begin{gather*} 0 \to \sk_{s+r-1} \Delta^p \to \Delta^p \to D_{\infty,s+r,s+r}^p \to 0 \\
0 \to \sk_{s-1} \Delta^p \to \sk_{s+r-1} \Delta^p \to D_{rss}^p \to 0 \end{gather*}
give isomorphisms
%The proof of Proposition~\ref{P:homologyofD} shows that the following maps are isomorphisms
\[ \xymatrix@R=.8pc{
H_{s+r} (D_{\infty,s+r,s+r}) \ar@{->}[r]_-\partial^-\cong &  H_{s+r-1} (\sk_{s+r-1} \Delta) \ar@{->}[r]^-\cong & H_{s+r-1} (D_{rss}).
}\]
Combined, these give isomorphisms
\begin{align*}
\vom_s &\to C( (H_s(D_{rss}) \ten H_s(D_{rss}) ) / \pi ) \\
\voc_{s,s+r} & \to C(H_s(D_{rss}) \ten H_{s+r-1} (D_{rss})) \\
\vom_{s+r,s+r} & \to C((H_{s+r-1}(D_{rss}) \ten H_{s+r-1}(D_{rss}) ) / \pi).
\end{align*}
\end{proof}

%%%%%%%%%%%%%%%%%%%%%%%%%%%%%%%%%%%%%%%%%%%%%%%%%%%%%%%%%%%%%%%%%%%%%%%%
%%%%%%%%%%%%%%%%%%%%%%%%%%%%%%%%%%%%%%%%%%%%%%%%%%%%%%%%%%%%%%%%%%%%%%%%
%%%%%%%%%%%%%%%%%%%%%%%%%%%%%%%%%%%%%%%%%%%%%%%%%%%%%%%%%%%%%%%%%%%%%%%%
%%%%%%%%%%%%%%%%%%%%%%%%%%%%%%%%%%%%%%%%%%%%%%%%%%%%%%%%%%%%%%%%%%%%%%%%
%%%%%%%%%%%%%%%%%%%%%%%%%%%%%%%%%%%%%%%%%%%%%%%%%%%%%%%%%%%%%%%%%%%%%%%%
%%%%%%%%%%%%%%%%%%%%%%%%%%%%%%%%%%%%%%%%%%%%%%%%%%%%%%%%%%%%%%%%%%%%%%%%

\begin{thm}\label{T:asymmetric} The cohomology of $\voc_{s,s'}$ is
\[ H^n \voc_{s,s'} = \begin{cases} \K & \text{if } n = s+s', \\ 0 & \text{otherwise.}\end{cases} \]
\end{thm}
\begin{proof}
Notice that $H_* D_{\infty s s}$ is concentrated in degree $s$ by Proposition~\ref{P:homologyofD}. So 
\begin{align*}
H^* \voc_{s,s'} &= H^* C(H_* D_{\infty ss} \ten H_* D_{\infty s's'}) \\ &= E^2(H_* D_{\infty ss} \ten H_* D_{\infty s's'} ) \\ &\cong E^2 (H_* D_{\infty ss}) \ten E^2( H_* D_{\infty s's'})\end{align*} where the last isomorphism is by Proposition~\ref{P:easye2}. The result follows from the computation of $E^1(D_{\infty ss})$ in Proposition~\ref{P:univexamplecalc}.% section~\ref{S:examplehomology}.
\end{proof}

%%%%%%%%%%%%%%%%%%%%%%%%%%%%%%%%%%%%%%%%
%%%%%%%%%%%%%%%%%%%%%%%%%%%%%%%%%%%%%%%%
\section{Cohomology of \texorpdfstring{$\vom$}{\unichar{937}-bar}}\label{S:cohomvom}
%%%%%%%%%%%%%%%%%%%%%%%%%%%%%%%%%%%%%%%%
%%%%%%%%%%%%%%%%%%%%%%%%%%%%%%%%%%%%%%%%
Fix $s\geq 0$. In this section we employ a short exact sequence in order to compute the cohomology of $\vom_s$. Let $A=\ker ( \voc_{s,s} \to \vom_s)$ and consider the exact sequence 
\[ 0 \to A \to \voc_{s,s} \to \vom_s \to 0. \]  We begin by identifying the complex $A$ and studying its cohomology.

Observe that for $p>s$, if $(\zeta, \zeta')\in \vo_{s,s}^p$ then $\sigma(\zeta,\zeta') \neq (\zeta, \zeta')$. Thus $\voc_{s,s}^p$ is a free $\K\pi$-module for $p>s$, so \[ \ker (\voc_{s,s}^p \to \vom_s^p) = A^p = (1+\sigma) \voc_{s,s}^p. \] Furthermore, $\vo_{s,s}^s = \set{(\id_{[s]}, \id_{[s]}) }$, so $A^s=0$.

We have now identified $A$ as the image of the map
\[ \voc_{s,s} \overset{1+\sigma}{\longrightarrow} \voc_{s,s}. \] 
The kernel of this map,
\[ \tvom =  \ker(1+\sigma: \voc_{s,s} \to A) \]
will be of independent interest 
%PREPRINTREF
(see \cite[Proposition 2.2]{me2b}). For now, notice that $\tvom^p = (1+\sigma)\voc^p_{s,s} = A^p$ for $p>s$ since $\voc^p_{s,s}$ is $\K\pi$-free. This implies that
\[ H^p \tvom = H^p A \qquad p \geq s+2. \]

Using this, we have the following lemma.

\begin{lem}\label{L:upsilonA}
Consider the inclusion $\tvom \to \voc_{s,s}$. If $s\geq 2$, then the induced map
\[ H^{2s} (\tvom) \to H^{2s} (\voc_{s,s}) \]
is zero. As a consequence, the inclusion $A \to \voc_{s,s}$ induces the zero map $H^{2s} (A) \to H^{2s} (\voc_{s,s})$.
\end{lem}
\begin{proof}
By assumption, $s\geq 2$, so  $H^{2s} \tvom = H^{2s} A$. 
The short exact sequence
\[ 0 \to \tvom \to \voc_{s,s} \to A \to 0 \]
induces the exact sequence
\[ \xymatrix@R=.8pc{
0 \ar@{->}[r] & H^{2s-1} A  \ar@{->}[r] & H^{2s} \tvom  \ar@{->}[r] & H^{2s} \voc_{s,s} \ar@{->}[r] \ar@{=}[d] & H^{2s} A \ar@{->}[r] & 0\\
&&& \K
}\] 
since $H^{2s} \voc_{s,s} = \K$ by Theorem~\ref{T:asymmetric} and $\tvom^{2s+1}=0$.
Seeking a contradiction, suppose that $H^{2s} \tvom \to H^{2s} \voc_{s,s}$ is nonzero. Then $0= H^{2s} A$, but since $s\geq 2$ we have $H^{2s} A = H^{2s} \tvom$. This contradicts our assumption.
\end{proof}

\begin{prop}\label{P:cohomatvom}
Fix $s\geq 0$. We have
\begin{align*}
  H^p(A) &= \begin{cases} \K & \text{for } p\in [s+1,2s] ,  s>0, \\ 0 & \text{else} \end{cases} &  H^p (\tvom) &= \begin{cases} \K & \text{for } p\in [s+2,2s] , s > 1, \\ \K & \text{for } p=0=s, \\ 0 & \text{else.}\end{cases}    
\end{align*}
\end{prop}
\begin{proof}
We use the long exact cohomology sequence associated to the short exact sequence
\[ 0 \to \tvom \to \voc_{s,s} \to A \to 0 \] as well as Theorem~\ref{T:asymmetric}, which says that $H^{2s} \voc_{s,s} = \K$ and $H^p \voc_{s,s} = 0$ for $p\neq 2s$.

If $s=0$ then $A=0$, so $\tvom = \voc_{0,0}$ and the result is obvious.

If $s=1$, then $A^p=0$ for $p\neq 2$ and $A^2\neq 0$, so exactness of 
\[ \xymatrix@R=.8pc{
0 \ar@{->}[r] &  H^{2} \tvom  \ar@{->}[r] & H^{2} \voc_{1,1} \ar@{->}[r] \ar@{=}[d] & H^{2} A \ar@{->}[r] \ar@{=}[d] & 0\\
&& \K & A^2
}\] 
implies $H^2 \tvom =0$ and $H^2A=\K$.

Suppose that $s\geq 2$. Examine the exact sequence
\[ 
0 \to H^{2s-1} A  \to H^{2s} \tvom  \to H^{2s} \voc_{s,s} \to  H^{2s} A \to   0.\]
By Lemma~\ref{L:upsilonA}, $H^{2s} \tvom \to H^{2s} \voc_{s,s}$ is zero. 
Thus we have $\K = H^{2s} \voc_{s,s} = H^{2s} A$, and $H^{2s-1} A \to H^{2s} \tvom = H^{2s} A = \K$ is an isomorphism.  For $p < 2s$, we have that 
\[ 0 \to H^{p-1} A \to H^p \tvom \to 0 \] is exact, so for $s+2 \leq p < 2s$ we have
\[ H^{p-1} A = H^p \tvom = H^p A = \K. \] To finish this case, notice that $A^{p-1} =0$ for $p-1 \leq s$, so 
$0 = H^{p-1} A = H^p \tvom$ for $p\leq s+1$.
\end{proof}

\begin{thm}\label{T:symmetric} The cohomology of $\vom_s$ is
\[ H^n \vom_s = \begin{cases} \K & n \in [s,2s] \\ 0 &\text{otherwise.} \end{cases} \]
\end{thm}
\begin{proof}
We use the exact sequence \[ 0 \to A \to \voc_{s,s} \to \vom_s \to 0,\] Theorem~\ref{T:asymmetric}, and Proposition~\ref{P:cohomatvom}. 
Notice immediately that $H^{i-1} \vom_s \cong H^i A$ for $i<2s$, so we are reduced to analyzing the exact sequence
\[ 0 \to H^{2s-1} \vom_s \to H^{2s} A \to H^{2s} \voc_{s,s} \to H^{2s} \vom_s \to 0. \]
We merely need to show that $H^{2s} A \to H^{2s} \voc_{s,s}$ is always zero.
When $s\geq 2$, 
%we saw at the beginning of the proof of Proposition~\ref{P:cohomatvom} that $H^{2s} A=H^{2s} \tvom \to H^{2s} \voc$ is zero. 
this is just Lemma~\ref{L:upsilonA}.
If $s=1$, and $\alpha \in A^2 = \tvom^2$ is a cycle, then $\alpha$ is a boundary in $\tvom$ since $H^2 \tvom =0$, hence $\alpha$ is a boundary in $H^2(\voc_{1,1})$. Finally, $H^0 A \to H^0 \voc_{0,0}$ is trivially zero since $A=0$ when $s=0$.
\end{proof}

%%%%%%%%%%%%%%%%%%%%%%%%%%%%%%%%%%%%%%%%%%%%%%%
\section{The \texorpdfstring{$E^2$}{E-2} page}\label{S:e2}
%%%%%%%%%%%%%%%%%%%%%%%%%%%%%%%%%%%%%%%%%%%%%%%

We now record the $E^2$ page of the spectral sequence. See Figure~\ref{F:e2orig}. 

\begin{figure}[ht]
\centering
%\scalebox{.7}
 {
\begin{tikzpicture}
    %Axes
%    \draw[->] (0.2,0) -- (-8,0);
	\draw[->] (-.65,0) -- (-8,0);		%new
	\draw (0.2,0) -- (-.35,0);		%new
	\draw (-.75,-.1) -- (-.55,.1);	%new
	\draw (-.45,-.1) -- (-.25,.1);	%new
    \draw[->] (0,-0.2) -- (0,5);
    %Horizontal ticks
    \draw (-1,2pt) -- (-1,-2pt) node[below=.01] {$-s$};
    \draw (-2.5,2pt) -- (-2.5,-2pt) node[below=.01] {$-s-r$};
    \draw (-4,2pt) -- (-4,-2pt) node[below=.01] {$-2s$};
    \draw (-5.5,2pt) -- (-5.5,-2pt) node[below=.01] {$-2s-r$};
    \draw (-7,2pt) -- (-7,-2pt) node[below left=.01 and -.6] {$-2s-2r$};
    %Horizontal ticks
    %\draw (-1,2pt) -- (-1,-2pt) node[below] {$s$};
    %\draw (-2.5,2pt) -- (-2.5,-2pt) node[below] {$s+r$};
    %\draw (-4,2pt) -- (-4,-2pt) node[below] {$2s$};
    %\draw (-5.5,2pt) -- (-5.5,-2pt) node[below] {$2s+r$};
    %\draw (-7,2pt) -- (-7,-2pt) node[below] {$2s+2r$};
    %Vertical Ticks
    \draw (2pt,1) -- (-2pt,1) node[right] {$2t$};
    \draw (2pt,2) -- (-2pt,2) node[right] {$2t+r-1$};
    \draw (2pt,3) -- (-2pt,3) node[right] {$2t+2r-2$};
    %Points
    \fill (-5.5,2) circle (2pt);
    %Arrow
    \draw [->,very thick] (-4,1) -- (-1,1) -- (-1,5);
    \draw [->,very thick] (-7,3) -- (-2.5,3) -- (-2.5,5);
\end{tikzpicture} 
}
\caption{$E^2(  \e(D_{rst}))$}\label{F:e2orig}
\end{figure}

\begin{thm}\label{T:e2page} Let $2\leq r \leq \infty$. Each bidegree in the $E^2$ page of the spectral sequence for $\e(D_{rst})$  is either $0$ or $\K$. The nontrivial terms are in
\begin{gather*}
\set{-s}\times [2t, \infty) \\
[-2s,-s-1] \times \set{2t} \\
\end{gather*}
and, if $r < \infty$, in
\begin{gather*}
\set{-s-r} \times [2t+2r-2, \infty) \\
\set{-2s-r} \times \set{2t+r-1} \\
[-2s-2r,-s-r-1] \times \set{2t+2r-2}.
\end{gather*}
% For $2\leq r \leq \infty$, the $E^2$ page of the spectral sequence for $\e(D_{rst})$ consists of $\K$ in the following ranges of bidegrees
% \begin{gather*}
% \set{-s}\times [2t, \infty) \\
% [-2s,-s-1] \times \set{2t} \\
% \end{gather*}
% and, if $r < \infty$, $\K$ in the following ranges of bidegrees
% \begin{gather*}
% \set{-s-r} \times [2t+2r-2, \infty) \\
% \set{-2s-r} \times \set{2t+r-1} \\
% [-2s-2r,-s-r-1] \times \set{2t+2r-2}
% \end{gather*}
% with zero everywhere else.
\end{thm}
\begin{proof}
Theorem~\ref{T:e1basis} gives the $E^1$ page. The structure of that page gives the ranges $\set{-s} \times (2t,\infty)$ and $\set{-s-r}\times (2t+2r-2,\infty)$. Theorems~\ref{T:asymmetric} and \ref{T:symmetric} combine with Proposition~\ref{P:comparison} to give the rest.
\end{proof}

Notice that when $r=\infty$, the structure of the $E^2$ page implies that all further differentials are zero, so
\[ E_{-p,q}^\infty = \begin{cases} \K & \text{if } q=2t \text{ and } -p\in [-2s,-s], \\ \K & \text{if } p=s \text{ and } q\geq 2t, \\ 0 & \text{else.}\end{cases} \]
We will compute the differentials for finite $r$ in the next two sections.

%%%%%%%%%%%%%%%%%%%%%%%%%%%%%%%%%%%%%%%%%%%%%%%
%\section{\texorpdfstring{$E^\infty=0$  when $r< \infty$}{E-infinity=0 when r<0}}
\section{\texorpdfstring{$E^\infty=0$  when $r< \infty$}{E-\Uinfty=0 when r<\Uinfty}}\label{S:einf}
%%%%%%%%%%%%%%%%%%%%%%%%%%%%%%%%%%%%%%%%%%%%%%%

Let $D=D_{rss}$ for $r<\infty$. The goal of this section is contained in its title: we wish to show that $E^\infty (\e(D)) = 0$.

\begin{lem}\label{L:ddfinite}
If $r$ is finite, then the bicomplex $C(D \ten D)$ is finite.
\end{lem}
\begin{proof}
The vector space $C(D\tensor D)^p_m$ has a basis consisting of elements $\ep \tensor \ep'$ where $\ep: [q] \hookrightarrow [p]$, $\ep': [q']\hookrightarrow [p]$, $q+q'=m$, and $[1,p] \ci \im \ep \cup \im \ep'$. Furthermore, since we are working in $D_{rss}$ we require that $q,q' \in [s,s+r-1]$. Thus we see that $C(D\ten D)^p_m$ is zero unless $m\in [2s,2(s+r-1)]$  and $p\in[s,2(s+r)]$, so $C(D\ten D)$ is bounded. Furthermore, each $C(D\ten D)^p_m$ is finite.
\end{proof}

\begin{prop}\label{P:justD}
For $r$ finite we have $H\Tot C(D\tensor D)=0$.
\end{prop}
\begin{proof}
Lemma~\ref{L:ddfinite} implies convergence, so we have
\[ H\Tot C(D\ten D) \cong \Tot E^\infty (D\ten D) \cong \Tot [E^\infty (D) \ten E^\infty (D)]\] by Proposition~\ref{P:easye2}. We saw in section~\ref{S:examplehomology} that $E^\infty(D_{rss}) =0$ for $r<\infty$. 
\end{proof}

We will need similar algebraic convergence results for $\e(D_{rst})$. 

\begin{thm}\label{T:algconv}
Let $2\leq r \leq \infty$. Then
\[ H_* \Tot C(\e(D_{rst})) \cong \Tot E^\infty (\e(D_{rst})). \]
\end{thm}
\begin{proof}
Write $B=C(\e(D_{rst}))$. 
%According to \cite[p.142]{weibel} the filtration on $TB$ is complete and exhaustive. 
Note that $\Tot B$ is the completion (as a filtered complex) of the sum totalization $(\Tot^\oplus B)_m = \oplus_{p+q=m} B_{p,q}$. % is $\Tot B$.
Reindexing so that we have a fourth quadrant, cohomological bicomplex with vertical differential $\hd$ and horizontal differential $\cd$, we have conditional convergence of our spectral sequence by \cite[Theorem 10.1]{boardman}.

Since each group $E^1_{-p,q}$ is a finite $\K$-module, only finitely many differentials $\ssd^j: E^j_{-p,q} \to E^j_{-p-j,q+j-1}$ may be nonzero. Thus we have strong convergence of our spectral sequence by \cite[Theorem 7.1]{boardman} and the Remark immediately following it. 
%
%We have
%\[ H_m (F^{-p}/F^{-p-1}) = 0 \qquad \qquad p> \begin{cases} 2s+2r & r<\infty \\ 2s & r=\infty \end{cases} \]
%This is a straightforward computation using Theorem~\ref{T:e1basis} (see figure on page~\pageref{F:e2orig}): 
%%in the case $r$ is finite and Proposition~\ref{P:e2inf} in the case $r=\infty$:
%\begin{align*}
%H_m F^{-p} / F^{-p-1} &= H_{m+p} C (\e(D_{\infty st}))^p \\ &= C H_{m+p} (\e(D_{\infty st}))^p \\ &= 0 &\text { for } p> 2s + 2r \text{ or } 2s.
%\end{align*}
%
%The short exact sequence of complexes
%\[ 0 \to F^{-p} / F^{-p-1} \to T(B) / F^{-p-1} \to T(B) / F^{-p} \to 0 \] gives that 
%\[H_m\left( T(B) / F^{-p-1} \right)  \to H_m\left( T(B) / F^{-p} \right) \] is an isomorphism for large $p$, so by Mittag-Leffler
%\[ {\varprojlim_p}^1 H_m\left( T(B) / F^{-p} \right) = 0.\]
%
%We then have the short exact sequence (see \cite[p.142 and 5.5.5]{weibel})
%\[ 0 \to {\varprojlim_p}^1 H_{m+1} \left( T(B) / F^{-p} \right) \to H_m(T(B)) \to \varprojlim_p H_m \left(T(B)/F^{-p} \right) \to 0 \]
%which implies
%\[ H_m(T(B)) = \varprojlim_p H_m \left(T(B)/F^{-p} \right) \] for all $m$.
\end{proof}

The following proposition works over any ground ring and, in particular, gives
\[ C(W\tp (D\ten D)) \cong W^v \tp C(D\ten D),\] where $W^v$ is the bicomplex concentrated on the $y$-axis with $W^v_{0,*} = W_*$.

\begin{prop}\label{P:bringoutvert}
Let $X_*$ be a chain complex and let $X^v_*$ be the bicomplex which has $X$ as its zeroth column. If $Y_*^\bullet$ is a cosimplicial chain complex then \[ C(X_* \tensor Y_*^\bullet) \cong X^v_* \tensor C(Y_*^\bullet). \]  \qed
\end{prop}

Finiteness of $C(D\tensor D)$ allows us to conclude that
\[ \Tot C(\e(D)) = \Tot C(W\tp (D\tensor D)) \cong W\tp \Tot C(D\tensor D). \]
Furthermore, the functor $W\tp -$ preserves quasi-isomorphism.

\begin{prop}\label{P:quasi}
Suppose that $L\to L'$ is a map of nonnegatively-graded $\K\pi$-complexes which induces an isomorphism in homology. Then \[ H(W\tp L) \to H(W\tp L') \] is an isomorphism as well.
\end{prop}
\begin{proof}
The K\"unneth spectral sequence (see \cite[Theorem 2.20]{mccleary})
\[ E_2^{p,q} = \bigoplus_{s+t=q} \Tor_{\K\pi}^p (H^s(W), H^t(L)) \Rightarrow H(W\tp L) \] is a first-quadrant spectral sequence, so it converges. The map $L\to L'$ induces an isomorphism on $E_2$.
\end{proof}

\begin{prop} For $r$ finite we have
$\Tot E^\infty (\e(D_{rst})) = H\Tot C(\e(D_{rst})) = 0$.
\end{prop}
\begin{proof} We already saw that $\Tot C(\e(D_{rst})) \cong W\tp \Tot C(D_{rst} \ten D_{rst})$, so we have $H\Tot C(\e(D_{rst})) =0$ by Propositions~\ref{P:justD} and \ref{P:quasi}.
The spectral sequence converges by Theorem~\ref{T:algconv}.
\end{proof}

%%%%%%%%%%%%%%%%%%%%%%%%%%%%%%%%%%%%%%%%%%%%%%%
\section{All Other Differentials Are Automatic}\label{S:autodiff}
%%%%%%%%%%%%%%%%%%%%%%%%%%%%%%%%%%%%%%%%%%%%%%%

A spectral sequence with $E^2$ page of the form of Theorem~\ref{T:e2page} with $E^\infty=0$ can only have one pattern of differential, which we give a rough picture of in Figure~\ref{F:differentials}. We need only consider differentials $\ssd^j: E^j_{p,q} \to E^j_{p-j,q+j-1}$ for $j\geq 2$.

\begin{prop}\label{P:nontrivialdiff}
The following differentials in the spectral sequence associated to $\e(D_{rst})$ are nontrivial:
\begin{align*}
\ssd^r: E^r_{-2s-r,2t+r-1} &\to E^r_{-2s-2r,2t+2r-2} \\
\ssd^{2r-1}: E^{2r-1}_{p,2t} &\to E^{2r-1}_{p-2r+1,2t+2r-2} & p&\in[-2s,-s-1] \\
\ssd^{2r-1-b}: E^{2r-1-b}_{-s,2t+b} &\to E^{2r-1-b}_{b+1-2r-s,2t+2r-2} & b&\in[0,r-2] \\
\ssd^r: E^r_{-s,2t+b} &\to E^{r}_{-s-r,2t+b+r-1} & b&\in [r-1,\infty)
\end{align*}
\end{prop}
\begin{proof}
First we look at the `top row' $[-2s-2r,-s-r] \times \set{2t+2r-2}$. All differentials $\ssd^j$ out of $E_{p,q}^j$ for $(p,q)\in [-2s-2r,-s-r] \times \set{2t+2r-2}$ must be zero. We list all possibilities for differentials mapping to this row which have the potential to be nontrivial:
\begin{align*} \ssd^{2r-1} &: E_{p,2t}^{2r-1} \to E^{2r-1}_{p-2r+1,2t+2r-2} \\ 
\ssd^r&: E_{-2s-r,2t+r-1}^r \to E^r_{-2s-2r,2t+2r-2} \\
\ssd^j&: E^j_{-s,2t+2r-j-1} \to E^j_{-s-j,2t+2r-2} \\
\end{align*}
where \begin{align*} p\in [-2s,-s]  &\leftrightsquigarrow p-2r+1\in [-2s-2r+1,-2r-s+1] \\ j\in [r,2r-2] &\leftrightsquigarrow -s-j \in [2-s-2r,-s-r]. \end{align*}
But we have that \[ [-2s-2r,-s-r] = [-2s-2r+1,-2r-s+1] \sqcup \set{-2s-2r} \sqcup [2-2r-s,-s-r], \] so each map listed above must have rank $1$.

This leaves us only with the leftmost column \[ \set{-s-r} \times [2t+2r-1,\infty)\] and part of the rightmost column $\set{-s} \times [ 2t+r,\infty)$ still unaccounted for. Then it is obvious that there is only one possibility:
\begin{align*} \ssd^r: E^r_{-s,q} &\to E^r_{-s-r,q+r-1}  & q\in [2t+r,\infty). \end{align*}
\end{proof}

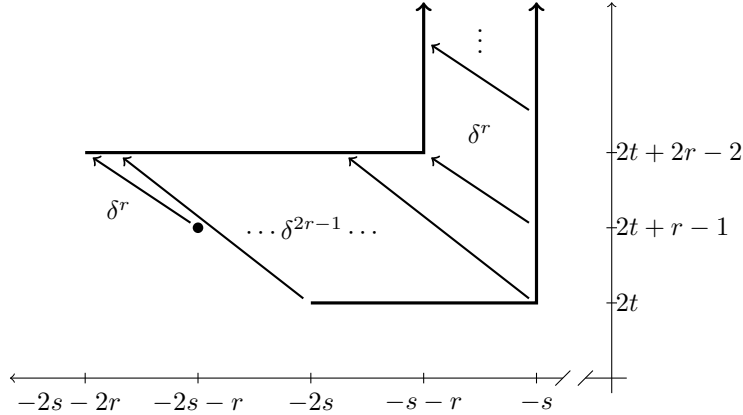
\begin{figure}[ht]
\centering
%\scalebox{.7} 
{
\begin{tikzpicture}
    %Axes
%    \draw[->] (0.2,0) -- (-8,0);
	\draw[->] (-.65,0) -- (-8,0);		%new
	\draw (0.2,0) -- (-.35,0);		%new
	\draw (-.75,-.1) -- (-.55,.1);	%new
	\draw (-.45,-.1) -- (-.25,.1);	%new
    \draw[->] (0,-0.2) -- (0,5);
    %Horizontal ticks
    \draw (-1,2pt) -- (-1,-2pt) node[below=.01] {$-s$};
    \draw (-2.5,2pt) -- (-2.5,-2pt) node[below=.01] {$-s-r$};
    \draw (-4,2pt) -- (-4,-2pt) node[below=.01] {$-2s$};
    \draw (-5.5,2pt) -- (-5.5,-2pt) node[below=.01] {$-2s-r$};
    \draw (-7,2pt) -- (-7,-2pt) node[below left=.01 and -.6] {$-2s-2r$};
    %Vertical Ticks
    \draw (2pt,1) -- (-2pt,1) node[right] {$2t$};
    \draw (2pt,2) -- (-2pt,2) node[right] {$2t+r-1$};
    \draw (2pt,3) -- (-2pt,3) node[right] {$2t+2r-2$};
    %Points
    \fill (-5.5,2) circle (2pt);
    %Arrow
    \draw [->,very thick] (-4,1) -- (-1,1) -- (-1,5);
    \draw [->,very thick] (-7,3) -- (-2.5,3) -- (-2.5,5);
    %Differentials!
%    \draw [->, thick] (-5.5,2) -- (-6.8,2.8667); %r
%    \draw [->, thick] (-1,2) -- (-2.3,2.8667);
    \draw [->, thick] (-5.6,2.0667) -- (-6.9,2.9333) node[pos=.5, below left=.01 and .01] {$\ssd^r$}; %Leftmost r
    \draw [->, thick] (-1.1,2.0667) -- (-2.4,2.9333) node[pos=.5, above=.5] {$\ssd^r$}; %Bottommost r
    \draw [->, thick] (-1.1,3.5667) -- (-2.4,4.4333) node[pos=.5, above=.2] {$\vdots$};    
    \draw [->, thick] (-4.1,1.06) -- (-6.5,2.9333) node[pos=.5, right=.3] {$\cdots \ssd^{2r-1} \cdots$}; %Leftmost 2r-1
    \draw [->, thick] (-1.1,1.06) -- (-3.5,2.9333); %Rightmost 2r-1
\end{tikzpicture} 
}
\caption{Differentials in the spectral sequence associated to $  \e(D_{rst})$}\label{F:differentials}
\end{figure}

\begin{cor}\label{C:vanishingofbidegrees}
Let $E_{*,*}^*$ be the spectral sequence associated with $\e(D_{rst})$. We record when various bidegrees become zero; they each contain a copy of $\K$ on the previous page. 
First for the lower right portion \begin{align*} 
E^{2r}_{p,2t} &=0 & p&\in[-2s,-s] \\ 
E^{2t+2r-v}_{-s,v}&=0 & v&\in [2t+1,2t+r-1] \\
E^{r+1}_{-s,q}&=0 & q&\in [2t+r,\infty)
\end{align*}
then for the upper left portion
\begin{align*}
E^{r+1}_{-2s-2r,2t+2r-2} &=0\\
E^{2r}_{p,2t+2r-2} &=0 & p&\in[-2s-2r+1,-2r-s+1] \\
E^{-s-p+1}_{p,2t+2r-2} &=0 & p&\in[-2r-s+2,-s-r] \\
E^{r+1}_{-s-r,q}&=0  & q&\in [2t+2r-1,\infty)
\end{align*}
and finally
\[ E^{r+1}_{-2s-r,2t+r-1} = 0. \] 
\end{cor}

\begin{rem} In the spectral sequence associated to $\e(D_{rst})$ we have
$E^{2r}=0$.
\end{rem}

%%%%%%%%%%%%%%%%%%%%%%%%%%%%%%%%%%%%%%%%%%%%%%%
\section{Products and Operations on Cycles}\label{S:products}
%%%%%%%%%%%%%%%%%%%%%%%%%%%%%%%%%%%%%%%%%%%%%%%

This section is a bit of a warm-up for what will come. The first goal is to define the (external) product in the spectral sequence of a cosimplicial chain complex $Y$ and show that it is commutative. We define external operations for \emph{$r$-cycles} and show that the bottom operation agrees with the external square.

In general, if $Y$ is a cosimplicial chain complex equipped with a multiplication
$ Y\ten Y \to Y$,
then there is a product
$ E^r(Y) \ten E^r(Y) \to E^r(Y) $
which is a derivation for $\ssd^r$, coming from 
\[ C(Y) \ten C(Y) \overset{AW}{\to} C(Y\ten Y) \to C(Y) \] 
where $AW$ is the Alexander-Whitney map from \eqref{E:AW}.
In our setting, we start with a cosimplicial map
\[ \theta: \e(Y) \to Y \]
and obtain a product by precomposition with the composite
\[ \inc: Y\ten Y = \K \ten Y \ten Y \overset{1\mapsto e_0}{\longrightarrow} W\ten Y \ten Y \to W\tp (Y\ten Y) = \e(Y). \]
The following proposition shows that products $E^r(Y)\ten E^r(Y) \to E^r(Y)$ obtained from
\[ Y\ten Y \overset{\inc}{\to} \e(Y) \overset{\theta}{\to} Y\]
are commutative for $r\geq 2$. 

\begin{prop}\label{P:commutative}
Let $Y$ be a cosimplicial chain complex and $r\geq 2$. The external product
\[ \mu_r: E^r(Y) \ten E^r(Y) \overset{AW}{\to} E^r(Y \ten Y) \to E^r(\e(Y)) \] is commutative.
\end{prop}
\begin{proof}
Since $AW$ becomes $\pi$-equivariant starting at $E^2$ (see the remark on page~\pageref{R:commutative}), we can reduce the problem to showing that the following holds on $E^2$:
\[ \mu_2 \sigma = \inc AW \sigma = \inc \sigma AW \overset{?}{=} \inc AW = \mu_2.\] The equality $\mu_r \sigma = \mu_r$ then follows for all $r\geq 2$.

Thus we merely need to show that $\inc \sigma = \inc$ on $E^2$. This is actually true on $E^1$. Consider $v\in Z_{-s}^1(Y\ten Y)\ci \Tot C(Y\ten Y)$, then we have the formula
\[ \partial (e_1 \ten v) = (1+\sigma) e_0 \ten v + e_1 \ten \partial v \] in $\Tot [W^v\tp C(Y\ten Y)] = \Tot C(\e(Y))$. Notice that
\[ \partial (e_1 \ten v) \in \partial F^{-s} = \partial Z^0_{-s+1-1} \ci B_{-s}^1 \]
and
\[ e_1 \ten \partial v \in F^{-s-1} = Z^0_{-s-1} \ci B_{-s}^1. \] Furthermore, 
\[ (1+\sigma) e_0 \ten v = (\inc + \inc \sigma) v,\] so $\inc = \inc \sigma$ on $E^1$.
\end{proof}

We now define external operations on $r$-cycles using the universal property of $D_{rst}$. The idea is that the lower `\reflectbox{L}' in the spectral sequence for $\e(D_{rst})$ should map to the external operations.
We saw in \S\ref{S:autodiff} that the $2^{\text{nd}}$ page is the same as the $r^{\text{th}}$ page in this spectral sequence. 

%\begin{defn}
Theorem~\ref{T:e2page} says that $E^2_{p,q}=E^r_{p,q}$ is either $\K$ or $0$; in the former case we write \begin{equation} 0\neq \exte_{p,q} \label{SYM:exte} \in E^2_{p,q}=E^r_{p,q} = \K\end{equation} for the generator.
%\end{defn}

\begin{defn}\label{D:preext}We define functions
\[ \preext^m: Z^r(Y) \to E^r(\e(Y)) \] as follows. For $y\in Z_{-s,t}^r (Y)$, we let 
\begin{equation}
\begin{aligned} \preext_v^{m}(y) &= E^r(\e(\dmap_y)) (\exte_{-s,m+t}) & m&\geq t\\
\preext_h^{m}(y) &= E^r(\e(\dmap_y)) (\exte_{m-s-t,2t}) &  m&\in[t-s,t]
 \end{aligned}\label{E:opdef}
\end{equation}
which are all classes of $E^r(\e(Y))$. Here $\dmap_y$ is the map from Proposition~\ref{P:universalprop} and $\exte_{p,q}$ is the universal class from \eqref{SYM:exte}.
\end{defn}

The idea is that the lower `\reflectbox{L}' in the spectral sequence for $\e(D_{rst})$ should map to the external operations of $y\in Z^r_{-s,t}$.

\begin{rem}
Recall that on $E^r$, $\dmap_y$ only depends on the class of $y$ in $E^r$, rather than on $y$ itself. The situation is much more subtle for $\e(\dmap_y)$, and, at $E^r$, this map \emph{does} depend on the specific choice of $r$-cycle.
\end{rem}

Notice that an $r$-cycle is, in particular, an $(r-1)$-cycle. Let us now compare the answers we get by considering an $r$-cycle in these two ways.
\begin{prop}\label{P:rtorminus1} Let $r>2$, and suppose $y \in Z^r_{-s,t}(Y)$. Write $y_r$ for $y$ considered as an element of $Z^r$ and $y_{r-1}\in Z^{r-1}$ for $y$ considered as an $r-1$ cycle. Then
\[ \preext^m(y_r) = [\preext^m(y_{r-1})]_r.\]
\end{prop}

Before beginning the proof, notice that we can easily compare the two constructions because
\begin{equation}\label{E:inclusion} \xymatrix{
D_{r-1,st} \ar@{->}[dr]_{\dmap_y} \ar@{->}[rr]^-{\dmap_\imath} && D_{rst} \ar@{->}[dl]^{\dmap_y} \\
& Y
}\end{equation}
commutes.

We will need the following lemma.  It says that if we consider the inclusion $\dmap_\imath: D_{rst} \to D_{\infty st}$, where $\imath$ is the class defined in \eqref{E:iota}, then $E^2(\e(\dmap_\imath))$ is an injection when restricted to the bottom `\reflectbox{L}'.
\begin{lem}\label{L:bottomL}
Consider the inclusion $\dmap_{\imath}: D_{rst} \to D_{\infty st}$. The map $E^2(\e(\dmap_\imath))$ 
is an isomorphism in the bidegrees 
%takes $\exte_{p,q}$ to $\exte_{p,q}$ for $(p,q)\in \big(
$[-2s, -s]\times \set{2t}$ and $\set{-s} \times [2t,\infty)$.
\end{lem}
\begin{proof}
On $E^1$, the map $\e(\dmap_\imath)$ is an isomorphism in this range. 
\end{proof}

\begin{proof}[Proof of Proposition~\ref{P:rtorminus1}]
%The proof is that the map
%\[ E^2(\e(D_{(r-1)st})) \to E^2(\e(D_{rst})) \] maps the lower `\reflectbox{L}' of the former onto the lower `\reflectbox{L}' of the latter. Comparing representatives, it is easy to see that $C(\ueio)$ induces an isomorphism of complexes at $E^1$ on the row $[-2s,-s]\times \set{2t}$ and gives an isomorphism on the column $\set{-s} \times [2t,\infty)$. The diagram
A special case of diagram (\ref{E:inclusion}) is when $Y=D_{\infty st}$ and $y=\imath$. Combined with Lemma~\ref{L:bottomL}, this tells us that  in the spectral sequence,
$\e(D_{r-1,st}) \to \e(D_{rst})$ takes the lower `\reflectbox{L}' to the lower `\reflectbox{L}'.  Furthermore, the following commutes,
\[ \xymatrix{
E^2(\e (D_{(r-1)st})) \ar@{->}[dr] \ar@{->}[rr]^-{E^2\e\dmap_\imath} && E^2(\e(D_{rst})) \ar@{->}[dl] \\
& E^2(\e(Y))
}\]  which implies that the representatives on the second page of $\preext^m(y_r)$ and $\preext^m(y_{r-1})$ are the same. The result follows.
\end{proof}

\subsection{Bottom Operation is the Square} We now show that the bottom operation coincides with the squaring operation. In particular, since the external product is commutative, this shows that the bottom operation is additive. %Recall the class $\imath$ from \eqref{E:iota}.

\begin{lem} Let $r\geq 2$. In the case of the universal example $D_{rst}$,
\[ \mu_r(\imath \ten \imath) \neq 0 \] where
\[ \mu_r: E^r (D_{rst}) \ten E^r (D_{rst}) \to E^r(\e(D_{rst}))\]
is the external multiplication.
\end{lem}
\begin{proof} First notice that we can factor the external multiplication as
\[ E^j (D_{rst}) \ten E^j (D_{rst})\overset{\cong}{\to} E^j(D_{rst} \ten D_{rst}) \to E^j(\e(D_{rst})), \] with the first arrow an isomorphism when $j\geq 2$ by Lemma~\ref{L:tensbicplx}.
We will prove that \[ E^2(D_{rst} \ten D_{rst}) \to E^2(\e( D_{rst})) \] is nontrivial; the result then follows by examining the spectral sequence for $\e(D_{rst})$ since nothing can hit the element in bidegree $(-2s,2t)$.

The vertical maps in the following commutative diagram are nontrivial, where $D_{rst} \to D_{\infty st}$ is the inclusion.
\[ \xymatrix{
E^2_{-2s,2t} (D_{rst} \ten D_{rst}) \ar@{->}[d] \ar@{->}[r] & E^2_{-2s,2t}(\e(D_{rst})) \ar@{->}[d] \\
E^2_{-2s,2t} (D_{\infty st} \ten D_{\infty st}) \ar@{->}[r] & E^2_{-2s,2t}(\e(D_{\infty st})) \\
}\] 
We must see that the bottom map is nontrivial. 
To do so, note that the diagram 
\[ \xymatrix{
E^2_{-2s,2t} (D_{\infty st} \ten D_{\infty st}) \ar@{->}[r] \ar@{=}[d] & E^2_{-2s,2t} (\e(D_{\infty st}))  \\
H^{2s} \voc_{s,s} \ar@{->}[r] & H^{2s} \vom_{s} \ar@{->}[u]^\cong_{\text{Prop.~\ref{P:comparison}}}
}\]
also commutes (where  $\voc_{s,s}$ and $\vom_s$ are defined in (\ref{E:voc}, \ref{E:vom})),
and the first map in the exact sequence 
\[ H^{2s} A \to H^{2s} \voc_{s,s} \to H^{2s} \vom_s \to 0 \] is zero by Lemma~\ref{L:upsilonA}, so $H^{2s} \voc_{s,s} \to H^{2s} \vom_s$ is an isomorphism.
\end{proof}

\begin{prop}\label{P:squarebottom}
Let $y\in Z^r_{-s,t}(Y)$, where $r\geq 2$. Then
\[ \mu_r([y],[y]) = \preext^{t-s}(y). \] 
\end{prop}
\begin{proof}
Let $\dmap_y: D_{rst} \to Y$ be the representing map from Proposition~\ref{P:universalprop}. Then \[ \preext^{t-s}(y) = E^r(\e(\dmap_y)) (\exte_{-2s,2t}) = E^r(\e(\dmap_y)) (\mu_r(\imath \ten \imath)) \] by the preceding lemma,
where $\exte_{p,q}$ is the universal class from \eqref{SYM:exte}.
%Our plan is to define the external operation on $y$ as the image of some class in $E^r(W\ten D_{rst} \ten D_{rst})$ under the map $E^r(\id \ten \dmap_y \ten \dmap_y)$. In particular, $\preext^{t-s}(y)$ is supposed to be the image of the element in bidegree $(-2s,2t)$. The preceding Lemma tells us that this element is $\mu_r (\imath \ten \imath)$.
Since $\mu_r: E^r(-)\ten E^r(-) \Rightarrow E^r(\e(-))$ is a natural transformation, we have a commutative diagram
\[ \xymatrix{
E^r(D_{rst})\ten E^r(D_{rst}) \ar@{->}[r]^-{\mu_r} \ar@{->}[d]_{E^r(\dmap_y) \ten E^r(\dmap_y)}& E^r(\e(D_{rst})) \ar@{->}[d]^{E^r(\e(\dmap_y)) }\\
E^r(Y)\ten E^r(Y) \ar@{->}[r]^-{\mu_r}& E^r(\e(Y))
}\]
so
\[ \preext^{t-s}(y) = \mu_r (E^r(\dmap_y) \ten E^r(\dmap_y) (\imath \ten \imath)) = \mu_r([y], [y]). \] 
\end{proof}

There are two consequences to this proposition. The first is that \[
\preext^{t-s}: Z_{-s,t}^r(Y) \to E^r(\e(Y)) \] is additive. This follows from commutativity of $\mu_r$. Second, $\preext^{t-s}$ induces a homomorphism
\[ \ext^{t-s}: E_{-s,t}^r(Y) \to E^r(\e(Y)) \]
since $\mu_r$ only depends on the $E^r$-class of a given $r$-cycle.

%%%%%%%%%%%%%%%%%%%%%%%%%%%%%%%%%%%%%%%%%%%%%%%
\section{Additivity and Sums of Bousfield-Kan Examples}\label{S:sumsofex}
%%%%%%%%%%%%%%%%%%%%%%%%%%%%%%%%%%%%%%%%%%%%%%%
The goal of this section is to prove the following proposition for $m>t-s$.
\begin{prop}[Additivity]\label{P:additivity} Let $r\geq 2$.
The functions
\begin{align*}
\preext_v^m:& Z_{-s,t}^r(Y) \to E^r_{-s,m+t}(\e(Y)) &  m&\geq t \\
\preext_h^m:& Z_{-s,t}^r(Y) \to E^r_{m-s-t,2t}(\e(Y)) & m&\in[t-s,t]
\end{align*} are homomorphisms.
\end{prop}

Let $x,y\in Z_{-s,t}^r(Y)$. The following diagram commutes:
\begin{equation} \xymatrix{
D_{rst} \ar@{->}[dr]_{\dmap_{x+y}} \ar@{->}[rr] && D_{rst} \oplus D_{rst} \ar@{->}[dl]^{\dmap_x \oplus \dmap_y} \\
& X 
} \label{E:additivediagram}\end{equation} where the top map is the diagonal. This suggests that analyzing the spectral sequence for $\e(D_{rst} \oplus D_{rst})$ may be helpful in understanding additivity. We will need greater generality later, so we now investigate the spectral sequence associated to $\e(D_{rst}\oplus D_{r's't'})$.

\begin{lem}\label{L:sumsplitting}
Let $X$ and $Y$ be cosimplicial chain complexes. Then
\[ \e( X\oplus Y)\cong \e(X) \oplus \e(Y) \oplus (W\ten X \ten Y) \]
via
\[ e_n \ten (x+y) \ten (x'+y') \mapsto \begin{gathered} e_n \ten x \ten x' + e_n\ten y\ten y' \\ + e_n \ten x \ten y' + \sigma e_n \ten x' \ten y \end{gathered} \]
and the obvious inclusions of the first two summands along with the inclusion 
\begin{align*}
W\ten X \ten Y &\to W\tp \Big( (X\oplus Y) \ten (X\oplus Y) \Big) \\
e_n\ten x \ten y & \mapsto e_n \ten x \ten y \\
\sigma e_n \ten x \ten y & \mapsto \sigma e_n \ten x\ten y = e_n \ten y \ten x.
\end{align*}
\end{lem}
\begin{proof} We may work with a fixed cosimplicial degree, so we prove the corresponding statement for chain complexes.
If $A$ and $B$ are chain complexes, then 
\[ \Big((A\oplus B) \ten (A\oplus B)\Big) = ( A\ten A)\oplus ( B \ten B) \oplus \big(A\ten B \oplus B\ten A\big) \] as $\K\pi$-modules. Since $A\ten B \oplus B\ten A$ is a free $\pi$-module, we see
\begin{multline*}  W\tp \Big( (A\oplus B) \ten (A\oplus B)\Big)  \\ = (W\tp (A \ten A)) \oplus (W\tp (B\ten B)) \oplus (W\ten A\ten B). \end{multline*}
\end{proof}

As mentioned above, we now consider the special case
\[  \e(D_{rst}\oplus D_{r's't'})
\cong \e(D_{rst}) \oplus \e(D_{r's't'}) \oplus (W\ten D_{rst}\ten D_{r's't'}).
\] 
Conormalization is an additive functor, as is the functor which takes a bicomplex to its associated spectral sequence, so we see that we need only compute the spectral sequence for 
\[ W\ten D_{rst} \ten D_{r's't'}. \] But this is easy -- the inclusion
\[ D_{rst} \ten D_{r's't'} \to W\ten D_{rst} \ten D_{r's't'}\] induces an isomorphism on $E^1$ by the K\"unneth theorem, so by Proposition~\ref{P:easye2},
\[ E^j(W\ten D_{rst} \ten D_{r's't'}) \cong E^j(D_{rst})\ten E^j(D_{r's't'})\] for $j\geq 2$.
In particular, $E^2(W\ten D_{rst} \ten D_{r's't'})$ is zero outside of the following set of bidegrees:
\begin{equation}\label{E:bidegrees} \left\{ \begin{gathered}
 (-s-s',t+t'),\\  (-s-s'-r,t+t'+r-1), (-s-s'-r',t+t'+r'-1), \\ (-s-s'-r-r',t+t'+r+r'-2)
\end{gathered} \right\}. \end{equation}

\begin{proof}[Proof of Proposition~\ref{P:additivity}]
Applying $\e$ to the diagram \eqref{E:additivediagram}, we have the following commutative diagram:
\[ \xymatrix{
\e( D_{rst}) \ar@{->}[dr]_{\e(\dmap_{x+y})} \ar@{->}[rr] && \e(D_{rst} \oplus D_{rst}) \ar@{->}[dl] \ar@{=}[r]^-{\ref{L:sumsplitting}} & \e(D_{rst}) \oplus \e(D_{rst}) \oplus (W\ten D_{rst} \ten D_{rst}) \ar@/^1.5pc/[dll]|{\e(\dmap_x) + \e(\dmap_y) + ???} \\ & \e(Y) }\] 
Using the formula from Lemma~\ref{L:sumsplitting} we see that the composite
\[ \e(D_{rst}) \to \e(D_{rst} \oplus D_{rst}) \twoheadrightarrow \e(D_{rst}) \oplus \e(D_{rst}) \] is just the diagonal.

We examine the map
\[ E^2(\e(D_{rst})) \to E^2 (W\ten D_{rst} \ten D_{rst}) \] in bidegrees
$\set{-s}\times [2t, \infty) $ and $[-2s,-s-1] \times \set{2t}$. Of course $E^2( W\ten D_{rst} \ten D_{rst}) = E^2(D_{rst}) \ten E^2(D_{rst})$ is zero at all of these bidegrees except for \[ E^2_{-2s,2t}(W\ten D_{rst} \ten D_{rst}) = \K. \] In particular, we know that  for $m>t-s$
\begin{align*} \preext^m_v(x+y) &= \preext^m_v(x) + \preext^m_v(y) & m\geq t \\
\preext^m_h(x+y) &= \preext^m_h(x) + \preext^m_h(y) & m\in (t-s,t]
\end{align*}

We saw at the end of the last section that $\preext^{t-s}$ is additive. \end{proof}

We will need the following lemma at the beginning of the next section.

\begin{lem}\label{L:combineddegen} We have
\[ E^2 (\e(D_{1st})) = 0, \] which implies \[ E^2\left( \e(D_{rst}\oplus D_{1s't'}) \right) \cong E^2 (\e(D_{rst})). \]
\end{lem}
\begin{proof}
The implication comes from \[ \e(D_{rst}\oplus D_{1s't'}) \cong \e(D_{rst}) \oplus \e(D_{1s't'}) \oplus (W\ten D_{rst} \ten D_{1s't'}) \]
and the fact that $E^2(W\ten D_{rst} \ten D_{1s't'}) \cong E^2(D_{rst}) \ten E^2(D_{1s't'}) =0$.

For the main statement, notice that $D_{1ss}$ is concentrated in homological degree $s$ and $E^0=E^1$ for this spectral sequence since
\[ H_k (D_{1ss}^p) = (D_{1ss}^p)_k = \begin{cases} \Delta^p_s & k=s \\ 0 & k \neq s. \end{cases} \]

The basis for the conormalization of $\e(D_{1ss})$ consists of 
\begin{gather*} e_n \ten \id_{[s]} \ten \id_{[s]} \\ e_n \ten d^0 \ten d^0
\end{gather*}
and also 
\[
e_0 \ten \ep \ten \ep'
\] where $\ep < \ep'$ in $\Delta^p_s$ and $[1,p] \ci \im \ep \cup \im \ep'$.
The differential $\ssd^1$ on a basis element in homological degree $2s$ is
\[ \ssd^1 (e_0 \ten \ep \ten \ep') = \begin{cases} e_0 \ten d^0\ep \ten d^0\ep' & \text{when } 0\in \im \ep \cup \im \ep', \\ 0 & \text{when } 0\notin \im \ep \cup \im \ep'. \end{cases}  \]
%(including the cases $S=T=[s] $ and $S=T=[1,s+1]\ci [s+1]$). 
We define a contracting homotopy $S$ with
\[ S(e_0 \ten \ep \ten \ep') = \begin{cases} 0  & \text{when } 0\in \im \ep \cup \im \ep', \\ e_0 \ten s^0\ep \ten s^0\ep', & \text{when } 0\notin \im \ep \cup \im \ep' \end{cases}  \]
which then satisfies $\ssd^1 S + S \ssd^1 = \id$.
%
%Using the contraction 
%\[ e_0 \ten \n(S) \ten \n(T) \mapsto \begin{cases} e_0 \ten \n(S-1) \ten \n(T-1) & 0\notin S\cup T \\ 0 & 0\in S\cup T  \end{cases} \] 
For the other degrees, we have
\[ \ssd^1 (e_n \ten \id_{[s]} \ten \id_{[s]}) = e_n \ten d^0 \ten d^0 \] so we see that $E^2=0$.
\end{proof}

% {P:unpisumdegen}  \label{P:E2landsonE2} 

%%%%%%%%%%%%%%%%%%%%%%%%%%%%%%%%%%%%%%%%%%%%%%%
%%%%%%%%%%%%%%%%%%%%%%%%%%%%%%%%%%%%%%%%%%%%%%%
\section{Definition of the Operations}\label{S:opdefinf}
%%%%%%%%%%%%%%%%%%%%%%%%%%%%%%%%%%%%%%%%%%%%%%%
%%%%%%%%%%%%%%%%%%%%%%%%%%%%%%%%%%%%%%%%%%%%%%%
In section~\ref{S:products} we defined (additive) operations $\preext^m$ on $r$-cycles. The goal of this section is to show that these induce operations which are defined on classes in the spectral sequence. 
The simplest thing would be to show that $\preext^m$ vanishes on
\[ B_{-s,t}^r = \partial Z_{-s+r-1,t-r+2}^{r-1} + Z^{r-1}_{-s-1,t+1} \]
for all $m$, but this does not happen. It turns out that the horizontal operations $\preext^m_h$ may be nonzero on $\partial Z_{-s+r-1,t-r+2}^{r-1}$, which leads to the indeterminacy in Theorem~\ref{T:inftymain}.

We begin with the easy part of $B_{-s,t}^r$: elements in lower filtration.
\begin{lem}\label{L:lowerfiltr}
The homomorphisms $\preext^m$ vanish on $Z^{r-1}_{-s-1,t+1}$ for $r\geq 2$.
\end{lem}
\begin{proof}
Write $r'=r-1, s'=s+1, t'=t+1$ and let $y\in Z^{r'}_{-s',t'}(Y) \ci Z^r_{s,t}(Y)$. 
Then the following commutes
\[ \xymatrix{
D_{rst} \ar@{->}[dr]_{\dmap_{y}^r} \ar@{->}[rr]^{\dmap_\imath} && D_{r's't'} \ar@{->}[dl]^{\dmap_y^{r'}} \\
& Y 
}\]
where, of course, we regard $\imath \in Z^{r'}_{-s',t'}(D_{r's't'})$ as an element of $Z^r_{-s,t}$.
If $r'\geq 2$, then by Theorem~\ref{T:e2page} the $E^2$ page of the spectral sequence of $\e(D_{r's't'})$ is zero except for the following ranges of bidegrees:
\begin{gather*}
\set{-s-1} \times [2t+2,\infty) \\
\set{-s-r} \times [2t+2r-2,\infty) \\
[-2s-2,-s-2] \times \set{2t+2} \\
\set{ -2s-r-1} \times \set{2t+r} \\
[-2s-2r,-s-r-1] \times \set{2t+2r-2}
\end{gather*}
which means that $E^r(\e(D_{r's't'}))$ is zero on the ranges we are interested in: \[ \set{-s}\times [2t,\infty) \text{ and } [-2s,-s-1]\times \set{2t}\]
Returning to the definition \eqref{E:opdef} of $\preext$, we note that the diagram
\[ \xymatrix{
E^r(\e(D_{rst})) \ar@{->}[dr] \ar@{->}[rr] && E^r(\e(D_{r's't'})) \ar@{->}[dl] \\
& E^r(\e(Y)) 
}\]
commutes and the rightmost composition takes $\exte_{p,q}$ to zero for \[ (p,q)\in \set{-s}\times [2t,\infty) \cup [-2s,-s-1]\times \set{2t},\] so all of the $\preext$ must vanish on $y$.

The case $r=2$ is even easier, since we know $E^2(\e(D_{1s't'})) = 0$ by Lemma~\ref{L:combineddegen}.

\end{proof}

We now shift our attention to $\partial Z_{-s+r-1,t-r+2}^{r-1}$. We run into a problem immediately, for we would like to use the diagram 
\[ \xymatrix{
D_{rst} \ar@{->}[dr]_-{\dmap_{\partial y}} \ar@{->}[rr]^-{\dmap_{\partial \imath}} && D_{r-1,s+1-r, r-t-2} \ar@{->}[dl]^-{\dmap_y} \\
& Y
}\]
where $y\in Z_{-s+r-1,t-r+2}^{r-1}$, but this diagram \emph{does not commute}. 
To see this, write 
\[ y= \sum_{j=s-r+1}^\infty y^j \qquad\text{ where } \qquad y^j \in C(Y)^j. \] We have 
\[ \partial y = \sum_{k=s}^\infty (\cd y^{k-1} + \hd y^{k}) \in F^{-s} \]
since $y$ is an $(r-1)$-cycle. Then
\[ C(\dmap_{\partial y}) \left(\sum_{j=s}^{s+r-1} \id_{[j]} \right) = \sum_{j=s}^{s+r-1} (\cd y^{j-1} + \hd y^j) \]
and
\begin{align*} C(\dmap_y) C(\dmap_{\partial \imath}) \left(\sum_{j=s}^{s+r-1} \id_{[j]} \right) &= C(\dmap_y) (\cd \id_{[s-1]}) \\
&= \cd y^{s-1}. \end{align*}
We can guarantee these expressions are equal when $y^s, y^{s+1}, \dots, y^{s+r-1}$ are all zero. We will need the following special case in the proof of Lemma~\ref{L:verticalpartial}.
\begin{lem}\label{L:partialdiagram} If \[ y= \sum_{j=s-r+1}^{s-1} y^j \in Z_{-s+r-1,t-r+2}^{r-1}\] then the diagram 
\[ \xymatrix{
D_{rst} \ar@{->}[dr]_-{\dmap_{\partial y}} \ar@{->}[rr]^-{\dmap_{\partial \imath}} && D_{r-1,s+1-r, r-t-2} \ar@{->}[dl]^-{\dmap_y} \\
& Y
}\]
commutes. 
\end{lem}

Suppose that
\[ y=\sum_{j=s-r+1}^\infty y^j \in Z_{-s+r-1,t-r+2}^{r-1}. \]
Then the tail $\sum_{j=s}^\infty y^j$ is in \[  F^{-s} \ci Z_{-s+r-1,t-r+2}^{r-1}, \]
so we can split $y$ up into two pieces
\[ y = \sum_{j=s-r+1}^{s-1} y^j + \sum_{j=s}^\infty y^j \]
both of which are in $Z_{-s+r-1,t-r+2}$. 
%We treat the tail piece now, so that we can use Lemma~\ref{L:partialdiagram} later.
We first treat the tail of $y$.
\begin{prop}
The homomorphisms $\preext^m$ vanish on $\partial F^{-s}$.
\end{prop}
\begin{proof}
Lemma~\ref{L:lowerfiltr} guarantees that the operations vanish on elements of $\partial F^{-s-1} \ci Z^{r-1}_{-s-1,t+1}$. By additivity we may consider the boundary of a single element in cosimplicial degree $s$:
\[ y= y_{t+1}^s. \]

Define a cosimplicial chain complex $\vsquare$ (depending on $s$ and $t$) 
schematically by
\[ \xymatrix{
\Delta_s^\bullet \ar@{->}[d]^{\id} & \text{in homological degree } t+1\\
\Delta_s^\bullet & \text{in homological degree } t
}\] with zero outside of these two homological degrees. The conormalization of this is pictured in the right hand side of Figure~\ref{F:dtov}, as well as a map of bicomplexes from $C(D_{rst})$ to $C(\vsquare)$ (open circles map to open circles).

%Define a cosimplicial chain complex $\vsquare$ (depending on $s$ and $t$; compare to $\mathbf{P}_{s,t+1}$ in \cite[p. 389]{bousfield}) so that $C(\vsquare)$ is a copy of $\K$ in bidegrees $(-s,t)$, $(-s-1,t)$, $(-s,t+1)$, and $(-s-1,t+1)$, with zero elsewhere. The horizontal and vertical differentials are nonzero whenever possible. The conormalization of $\vsquare$ is pictured in the right hand side of Figure~\ref{F:dtov}, as well as a map of bicomplexes from $C(D_{rst})$ to $C(\vsquare)$ (open circles map to open circles).

\begin{figure}[bh]
\centering
\begin{tikzpicture}
    %Axes
    \draw[->] (0.2,0) -- (-4,0);
    \draw[->] (0,-0.2) -- (0,4);
    %Horizontal ticks
    \draw (-.5,2pt) -- (-.5,-2pt) node[below] {$-s$};
    \draw (-1,2pt) -- (-1,-2pt); % node[below] {$-(s+r)$};
    \draw (-2.5,2pt) -- (-2.5,-2pt); 
    \draw (-3,2pt) -- (-3,-2pt); %node[below] {$-v$}; 
    %Vertical Ticks
    \draw (2pt,.5) -- (-2pt,.5) node[right] {$t$};
    \draw (2pt,1) -- (-2pt,1) node[right] {$t+1$};
    \draw (2pt,2.5) -- (-2pt,2.5) node[right] {$t+r-2$};
    \draw (2pt,3) -- (-2pt,3) node[right] {$t+r-1$};
    %Points on x=-y
    \draw (-.5,.5) circle (2pt);
    \draw (-1,1) circle (2pt);
    \fill (-2.5,2.5) circle (2pt);
    \fill (-3,3) circle (2pt);
    %Points on x=-y-1
    \draw (-1,.5) circle (2pt);
    \fill (-1.5, 1) circle (2pt);
    \fill (-2.5,2) circle (2pt);
    \fill (-3,2.5) circle (2pt);
    \fill (-3.5,3) circle (2pt);
    %Horizontal Arrows
    \draw [->] (-.6,.5) -- (-.9,.5);
    \draw [->] (-1.1,1) -- (-1.4,1);
    \draw [->] (-2.1,2) -- (-2.4,2);
    \draw [->] (-2.6,2.5) -- (-2.9,2.5);
    \draw [->] (-3.1,3) -- (-3.4,3);    
    %Vertical Arrows
    \draw [->] (-1,.9) -- (-1,.6);
    \draw [->] (-1.5,1.4) -- (-1.5,1.1);
    \draw [->] (-2.5,2.4) -- (-2.5,2.1);
    \draw [->] (-3,2.9) -- (-3,2.6);
    %Dotted Lines
    \draw [dotted] (-1.5,1.5)--(-2,2);
%    \draw [dotted] (-1.5,1)--(-2.5,2);
    \draw [dotted] (-1.8,1.3)--(-2.2,1.7);
    %%%%%%%%%%%%%% NEW GUY
        %Axes
    \draw[->] (5.2,0) -- (3,0);
    \draw[->] (5,-0.2) -- (5,4);
    %Horizontal ticks
    \draw (4.5,2pt) -- (4.5,-2pt) node[below] {$-s$};
    \draw (4,2pt) -- (4,-2pt); % node[below] {$-(s+r)$};
    %Vertical Ticks
    \draw (5cm + 2pt,.5) -- (5cm -2pt,.5) node[right] {$t$};
    \draw (5cm + 2pt,1) -- (5cm-2pt,1) node[right] {$t+1$};
    %Points on x=-y
    \draw (4.5,.5) circle (2pt);
    \draw (4,1) circle (2pt);
    %Points on $x=-y+1
    \fill (4.5,1) circle (2pt);
    %Points on x=-y-1
    \draw (4,.5) circle (2pt);
    %Horizontal Arrows
    \draw [->] (4.4,.5) -- (4.1,.5);   
    \draw [->] (4.4,1) -- (4.1,1); 
    %Vertical Arrows
    \draw [->] (4,.9) -- (4,.6);
    \draw [->] (4.5,.9) -- (4.5,.6);
    %Diagram Map
    %\draw [->, very thick] (1,2) arc (120:50:1.5cm);
    \draw [->, very thick] (1,1.5) arc (110:70:2cm);
\end{tikzpicture}
\caption{$C(D_{rst}) \to C(\vsquare)$}\label{F:dtov}
\end{figure}
This figure tells us that the diagram
\[ \xymatrix{ 
C(D_{rst}) \ar@{->}[rr]\ar@{->}[dr]_-{C(\dmap_{\partial z})} && C(\vsquare) \ar@{->}[dl] \\
& C(Y) 
}\] commutes, where
$C(\vsquare) \to C(Y)$ is the map taking the square in the former to the square
\[ \xymatrix{
 \cd y^s_{t+1} \ar@{->}[d] & y^s_{t+1} \ar@{->}[d] \ar@{->}[l] \\
 \cd\hd y^s_{t+1} & \hd y^s_{t+1} \ar@{->}[l]
}\]
in $C(Y)$. 
We apply $\e$ to get the commutative diagram
\[ \xymatrix{ 
\e(D_{rst}) \ar@{->}[rr]\ar@{->}[dr]_-{\e(\dmap_{\partial y})} && \e(\vsquare) \ar@{->}[dl] \\
& \e(Y) 
}\] The vanishing of the vertical homology of $\vsquare$ implies the vanishing of $E^1(\e(\vsquare))$, so $E^1(\e(\dmap_{\partial y}))=0$.
\end{proof}

\begin{lem}\label{L:verticalpartial}
The vertical maps $\preext_v$ vanish on $\partial Z_{-s+r-1,t-r+2}^{r-1}$ for $r>2$.
\end{lem}
\begin{proof}
Let $r'=r-1$, $s'=s+1-r$, $t'=t-r+2$ and suppose that $y\in Z_{-s+r-1,t-r+2}^{r-1}$ has the form \[ y = \sum_{j=s-r+1}^{s-1} y^j, \] an assumption that  the previous proposition and additivity allow us to make. Since $y$ has this form, the following diagram commutes by Lemma~\ref{L:partialdiagram}.
\[ \xymatrix{
D_{rst} \ar@{->}[dr]_{\dmap_{\partial y}} \ar@{->}[rr]^{\dmap_{\partial\imath}} && D_{r's't'} \ar@{->}[dl]^{\dmap_y} \\
& Y 
}\]
Applying Theorem~\ref{T:e2page} to $\e(D_{r's't'})$ (when $r>2$) we find nonzero terms exactly in the following bidegrees:
\begin{gather*}
\set{-s-1-r}\times [2t-2r+4,\infty) \\
\set{-s}\times [2t,\infty) \\
[-2s+2r-2,-s+r-2] \times \set{2t-2r+4} \\
\set{-2s+2r-1} \times \set{2t-r+2} \\
[-2s,-s-1] \times \set{2t}.
\end{gather*}
The column $\set{-s} \times [2t,\infty)$ vanishes at $E^r$ by the `upper left portion' part of Corollary~\ref{C:vanishingofbidegrees}, so the vertical operations $\preext_v(\partial y)$ vanish at $E^r$.
\end{proof}

\begin{lem}\label{L:vanishingr2}
If $r=2$ then the homomorphisms $\preext$ vanish on $\partial Z_{-s+1,t}^{1}$.
\end{lem}
\begin{proof}
Apply $E^2 \e$ to the diagram from Lemma~\ref{L:partialdiagram}:
\[ \xymatrix{
D_{2st} \ar@{->}[dr]_{\dmap_{\partial y}} \ar@{->}[rr]^{\dmap_{\partial\imath}} && D_{1s't'} \ar@{->}[dl]^{\dmap_y} \\
& Y
}\]
Lemma~\ref{L:combineddegen} says that $E^2 ( \e (D_{1s't'}))=0$.
\end{proof}

\begin{thm}\label{T:inftymain} The homomorphisms of Proposition~\ref{P:additivity} induce homomorphisms
\begin{align*}
\ext_v^m:& E_{-s,t}^r(Y) \to E^r_{-s,m+t}(\e(Y)) &  m&\geq t \\
\ext_h^m:& E_{-s,t}^r(Y) \to E^w_{m-s-t,2t}(\e(Y)) & m&\in[t-s,t]
\end{align*}
where \[ w = \begin{cases} r & m=t-s \\ 2r-2 & m\in [t-s+1,t-r+2] \\ r+t-m & m\in [t-r+3, t]. \end{cases} \] Notice that if $r=2$, then $w=2$.
\end{thm}
\begin{proof} We have already shown that the vertical operations pass to this quotient using Lemmas~\ref{L:lowerfiltr} and \ref{L:verticalpartial}. The well-definedness of the horizontal operations follows from the diagram in Lemma~\ref{L:verticalpartial} by applying the second part of Corollary~\ref{C:vanishingofbidegrees} to $D_{r-1,s+1-r, t-r+2}$.
\end{proof}

We give an example to show that the $w$ above is the best possible. Consider the diagram \[ \xymatrix{
\e(D_{rst}) \ar@{->}[dr]_{\e(\dmap_{\partial y})} \ar@{->}[rr]^-{\e(\dmap_{\partial \imath})} && \e(D_{r-1,s+1-r,t-r+2}) \ar@{->}[dl]^{\e(\dmap_y)} \\
& \e(Y) 
}\]
from Lemma~\ref{L:verticalpartial}.
\begin{ex}\label{exampleonlastpage}
We let $t=s\geq r-1$, and take as our example the cosimplicial chain complex $Y=D_{r-1,s+1-r,s-r+2}$ together with the class\[ y=\imath= \sum_{k=s+1-r}^k \id_{[k]}.\] 
The diagram
\[ \xymatrix{
H_s(D_{rss}) \ar@{->}[r]^-{\dmap_{\partial \imath}} \ar@{->}[d]^\cong& H_{s-1}(D_{r-1,s+1-r,s-r+2}) \\
H_s (D_{\infty ss}) \ar@{->}[r]^\cong_\partial & H_{s-1} (\sk_{s-1}  \Delta) \ar@{->}[u]^{\cong} 
}\]
commutes, and so the diagram
\[ \xymatrix{
& \vom_{s} \ar@{->}[dr] \ar@{->}[dl] &\\ 
E^1(\e(D_{rss})) \ar@{->}[rr] && E^1(\e( D_{r-1,s+1-r,s-r+2}))
}\]
also commutes, where $\vom_{s}$ is the complex defined in \eqref{E:vom}. Thus, at $E^2$, generators in the strip $[-2s,-s]\times \set{2t}$ map to nonzero elements. Vanishing of their images occurs at exactly the page described by `$w$' in the theorem.
\end{ex}

\section*{Acknowledgments} This work was part of the author's Ph.D. thesis. The author thanks his advisor, Jim McClure, for careful readings and numerous clarifying suggestions. The author would also like to thank Sean Tilson for his interest and for useful discussions about certain parts of this paper. Finally, the many suggestions from the anonymous referee helped the author improve the paper substantially, and for that the author thanks the referee as well.

\bibliographystyle{plain}
\bibliography{operations}
\end{document}